\documentclass[11pt]{article}

\usepackage{amsmath,amsthm}

\usepackage{amssymb,latexsym}

\usepackage{enumerate}

\usepackage{amsmath}
\usepackage{amssymb}
\newcommand{\R}{{\mathbb R}}

\newcommand{\sgn}{\mbox{\rm sgn}}

\newcommand{\N}{{\mathbb N}}
\newcommand{\LL}{{\mathbb L}}

\newcommand{\DD}{{\cal D}}
\newcommand{\HH}{{\cal H}}

\newcommand{\PP}{{\cal P}}

\newcommand{\FF}{{\cal F}}
\newcommand{\SSS}{{\cal S}}

\newcommand{\arrowp}{\mathop{\rightarrow}_{P}}
\newcommand{\Rd}{{{{\mathbb R}^d}}}

\newcommand{\dist}{\mbox{\rm dist}}

\newtheorem{theorem}{\bf Theorem}[subsection]
\newtheorem{proposition}[theorem]{\bf Proposition}%[subsection]
\newtheorem{lemma}[theorem]{\bf Lemma}%[subsection]
\newtheorem{corollary}[theorem]{\bf Corollary}%[subsection]

\theoremstyle{definition}
\newtheorem*{definition}{Definition}

\newtheorem{remark}[theorem]{Remark}

\newcommand{\nsubsection}{\setcounter{equation}{0}\subsection}

\begin{document}
\title{Reflected BSDEs in time-dependent  convex regions}
\author{Tomasz Klimsiak, Andrzej Rozkosz\footnote{Corresponding
author. Tel.: +48-56 6112953; fax: +48-56 6112987.}\ \ and Leszek
S\l omi\'nski
\mbox{}\\[2mm]
{\small  Faculty of Mathematics and Computer Science,
Nicolaus Copernicus University},\\
{\small Chopina 12/18, 87-100 Toru\'n, Poland}}
\date{}
\maketitle
\begin{abstract}
We prove the existence and  uniqueness of solutions of reflected
backward stochastic differential equations in time-dependent
adapted and c\`adl\`ag convex regions $\DD=\{D_t;t\in[0,T]\}$. We
also show that the solution may be approximated by solutions of
backward equations with reflection in appropriately defined
discretizations of $\DD$ and by a modified penalization method.
The approximation results are new even in the one-dimensional
case.
\end{abstract}

\footnotetext{{\em Email addresses:} tomas@mat.umk.pl (T.
Klimsiak), rozkosz@mat.umk.pl (A. Rozkosz), leszeks@mat.umk.pl (L.
S\l omi\'nski). }

%\footnotetext{{\em MSC 2000 subject classifications.} 60H10, 60H20}

\nsubsection{Introduction} \label{sec1}

In the present paper we investigate the problems of existence,
uniqueness and approximation of solutions to multidimensional
backward stochastic differential equations (BSDEs for short) with
reflection in time-dependent
random convex regions. % $\DD=\{D_t,t\in[0,T]\}$ such that
%$t\mapsto D_t$ is c\`adl\`ag with respect to the Hausdorff metric.

In the one-dimensional case these problems are quite well
investigated. In the pioneering paper \cite{EKPPQ} reflected BSDEs
(RBSDEs) with one continuous barrier and Lipschitz continuous
coefficient are thoroughly investigated. Subsequently the results
of \cite{EKPPQ} were generalized to equations with possibly
discontinuous barrier or coefficient satisfying less restrictive
regularity or growth conditions (see, e.g., \cite{Kl:EJP,RS2} and
the references therein). One-dimensional RBSDEs with two
continuous  reflecting barriers were first studied in \cite{CK}.
Recent results for equations with two possibly discontinuous
barriers are to be found in \cite{HHO,Kl:BSM,PX}.

Existence, uniqueness and approximation by the penalization method
of multidimensional RBSDEs were for the first time studied in
\cite{GP} in the case of fixed convex domain. In \cite{BEO,PR} the
existence and uniqueness results of \cite{GP} were generalized to
equations involving subdifferential of a fixed proper convex
lower-semicontinuous function.

Our main goal is to generalize the results of \cite{GP} to the
case of time-dependent random regions and at the same time
generalize to the multidimensional case some one-dimensional
results proved in \cite{CK,EKPPQ,Kl:EJP,Kl:BSM,RS2,PX} for
continuous barriers or discontinuous barries satisfying the
so-called Mokobodzki condition.

In the paper, like in \cite{GP} and all the papers mentioned
above, we assume that the underlying filtration is generated by a
Wiener process. Reflected equations with underlying filtration
generated by a Wiener process and an independent random Poisson
measure are considered in \cite{HO,O}. In \cite{HO} existence and
uniqueness results for one-dimensional equations with one and two
reflecting c\`adl\`ag barriers are proved, while in \cite{O} the
multidimensional  results of \cite{GP} are generalized to
Wiener-Poisson type equations in fixed bounded convex domain in
$\R^d$.

We now describe more precisely the content of the paper. Let $W$
be a standard $d$-dimensional Wiener process and let $(\FF_t)$
denote the standard augmentation of the natural filtration
generated by $W$. Suppose we are given a family
$\DD=\{D_t;\,t\in[0,T]\}$ of time-dependent random closed convex
subsets of $\R^m$ with nonempty interiors, an $\FF_T$-measurable
random vector $\xi=(\xi^1,\dots,\xi^m)$ with values in $D_T$ (the
terminal value) and a measurable function
$f:[0,T]\times\Omega\times\R^m\times\R^{m\times d}\to\R^m$
(coefficient). In the paper we consider RBSDEs in $\DD$ of the
form
\begin{equation}
\label{eq1.1}
Y_t=\xi+\int^T_tf(s,Y_s,Z_s)\,ds -\int^T_tZ_s\,dW_s+K_T-K_t,\quad
t\in[0,T].
\end{equation}
By a solution to (\ref{eq1.1}) we understand
a triple $(Y,Z,K)$ of $(\FF_t)$-adapted processes such that
$Y_t\in D_t$ for $t\in[0,T]$, $K$ is a process of locally
bounded variation $|K|$ increasing only when $Y_t\in\partial D_t$
and (\ref{eq1.1}) is satisfied.

In \cite{GP} it is proved  that if $\xi\in L^2$,
$\int_0^T|f(s,0,0)|^2\,ds\in L^1$, $f$ is Lipschitz continuous in
both variables $y,z$ and $D_t=G$, $t\in[0,T]$, where $G$ is a
nonrandom convex set with nonempty interior then there exists a
unique solution $(Y,Z,K)$ of (\ref{eq1.1}) such that $Y,K$ are
continuous, $Y^*_T,K^*_T\in L^2$ and $Z\in \PP^2$, i.e. $Z$ is
progressively measurable and $(\int^T_0\|Z_t\|^2\,dt)^{1/2}\in
L^2$ (Here and in the sequel we use the notation
$X^*_t=\sup_{s\leq t} |X|_s$, $t\in[0,T]$). In the present paper
we make  the same assumptions on the terminal value $\xi$ and
coefficient $f$. Our main assumption on $\DD$ says that the
process $t\mapsto D_t$ is $(\FF_t)$-adapted and c\`adl\`ag with
respect to the Hausdorff metric, and one can find a semimartingale
$A$ of the class $\HH^2$ (see Section \ref{sec2} for the
definition) such that $A_t\in\mbox{\rm Int} D_t$ for $t\in[0,T]$
and $ \inf_{t\leq T} \mbox{\rm dist} (A_t,\partial D_t)>0$. The
last condition is an analogue of the so-called Mokobodzki
condition considered up to now only in the one-dimensional case
(see \cite{Kl:BSM,PX} and the references therein). In our main
theorem we prove that under the above assumptions on $\xi,f,\DD$
there exists a unique solution $(Y,Z,K)$ of (\ref{eq1.1}) such
that $Y,K$ are c\`adl\`ag and $Y,Z,K$ have the same integrability
properties as in \cite{GP}, i.e. $Y^*_T,K^*_T\in L^2$, $Z\in
\PP^2$. If, in addition, $t\mapsto D_t$ is continuous, then $Y,K$
are continuous. Therefore our theorem generalizes the results of
\cite{GP} to time-dependent regions and the same time generalizes
one-dimensional results with time-dependent barriers to the
multidimensional case. But let us note that in the one-dimensional
case one can also prove existence and uniqueness of solutions of
RBSDEs with two reflecting barriers $L,U$ such that
$\DD=\{(L_t,U_t),t\in[0,T]\}$ does not satisfy the Mokobodzki
condition (see \cite{Kl:BSM}) and with less restrictive
assumptions on $\xi,f$ (see \cite{HHO,Kl:BSM}).  In general, these
solutions have weaker integrability properties.

The uniqueness of solutions of (\ref{eq1.1})  can be proved by
some  modification of known methods. The idea behind our proof of
existence is as follows. We consider piece-wise constant
time-dependent processes $\DD^j$ such that $\DD^j\rightarrow\DD$
in the Hausdorff metric uniformy in probability. Then we prove
that on each random interval on which $\DD^j$ is a constant random
set there exists a unique solution of some local RBSDE. Piecing
the local solutions together we obtain a solution $(Y^j,Z^j,K^j)$
of (\ref{eq1.1}) in $\DD^j$. Finally, we show that the sequence
$\{(Y^j,Z^j,K^j)\}$ converges as $j\rightarrow\infty$ and its
limit is a solution of (\ref{eq1.1}) in $\DD$. The  method
described above is new even in the one-dimensional case. To our
knowledge the results on existence and uniqueness of local
solutions in random convex sets are also new.

We also consider approximation of solutions of (\ref{eq1.1}) by
the penalization method. This method proved to be useful in the
case of  one-dimensional RBSDEs with regular and irregular
barriers (see, e.g., \cite{EKPPQ,Kl:SPA,Kl:EJP,Kl:BSM,LX,RS2}). In
\cite{LX} it is observed that in the last case, i.e. if the
barriers are discontinuous, the usual method provides only
pointwise approximation of the first component $Y$ and weak
approximation of $K$ and the martingale part of the solution.  To
generalize the penalization method to the irregular
multidimesional case  and at the same time to get uniform
approximation of $Y$ and strong approximation of $K$ and the
martingale part we consider a modified scheme. It has the form
\begin{equation}
\label{eq1.2}
Y^n_t=\xi+\int_t^Tf(s,Y^n_s,Z^n_s)\,ds-\int_t^TZ^n_s\,dW_s
+K^n_T-K^n_t,
\end{equation}
where
\begin{align}
\label{eq1.3} K^n_t&=-n\int_0^t (Y^n_s-\Pi_{D_s}(Y^n_s))\,ds
-\sum_{0<\sigma_{n,i}\leq
t}(Y^n_{\sigma_{n,i}}-\Pi_{D_{{\sigma_{n,i}}-}}
(Y^n_{\sigma_{n,i}}))\nonumber \\
&\equiv K^{n,c}_t+K^{n,d}_t,\quad t\in[0,T]
\end{align}
with $\sigma_{n,0}=0$,
$\sigma_{n,i}=\inf\{t>\sigma_{n,i-1};\,\rho(D_t\cap B(0,n),D_{t-}
\cap B(0,n))>1/n\}\wedge T$, $i=1,\dots,k_n$, where $k_n$ is
chosen so that $P(\sigma_{n,k_n}<T)\to 0$ as $n\rightarrow\infty$
(Here $B(0,n)=\{x\in\Rd;\,|x|\leq n\}$, $n\in\N$). Note that $K^n$
is a c\`adl\`ag process of locally bounded variation such that
\[
K^n_0=0,\quad\Delta K^n_{\sigma_{n,i}}
=\Pi_{D_{\sigma_{n,i}-}}(Y^n_{\sigma_{n,i}})-Y^n_{\sigma_{n,i}}
=-\Delta Y^n_{\sigma_{n,i}}.
\]
In fact, on  each interval $[\sigma_{n,i-1},\sigma_{n,i})$,
$i=1,\dots,{k_n+1}$, where $\sigma_{n,{k_n+1}}=T$, the pair
$(Y^n,Z^n)$ is a solution of the classical BSDEs with  Lipschitz
coefficients of the form
\begin{align*}
Y^n_t&=\Pi_{D_{\sigma_{n,i}-}}(Y^n_{\sigma_{n,i}})
+\int^{\sigma_{n,i}}_tf(s,Y^n_s,Z^n_s)\,ds
-\int^{\sigma_{n,i}}_tZ^n_s\,dW_s\\
&\quad -n\int^{\sigma_{n,i}}_t (Y^n_s-\Pi_{D_s}(Y^n_s))\,ds,\quad
t\in[\sigma_{n,i-1},\sigma_{n,i}).
\end{align*}
Notice that as compared with the usual penalization method, in the
penalization term $K^n$ the discontinuous part $K^{n,d}$ appears.
If the mapping $t\mapsto D_t$ is continuous then $K^{n}=K^{n,c}$,
so (\ref{eq1.2}), (\ref{eq1.3}) reduce to the usual penalization
scheme. We show that under the above-mentioned assumptions under
which there exists a unique solution $(Y,Z,K)$ of (\ref{eq1.1}),
it is a limit in probability of $\{(Y^n,Z^n,K^n)\}$ in the space
$\SSS\times\PP\times\SSS$ (see Section \ref{sec2} for its
definition). This result is new even for one-dimensional RBSDEs
with one discontinuous barrier.

It is known that one can use RBSDE to investigate viscosity
solutions (see \cite{EKPPQ,MPRZ,PR}) or weak solutions (see
\cite{Kl:SPA,Kl:PA,KR:JEE,RS1}) of variational inequalities. In
fact, this work was intended as the first step to investigate by
probabilistic methods this sort of problems for systems with
time-dependent constraints. These problems, however, will be
studied elsewhere.

\nsubsection{Notation and preliminary estimates}
\label{sec2}

For $x\in\R^m$, $z\in\R^{m\times d}$ we set
$|x|^2=\sum^m_{i=1}|x_i|^2$, $\|z\|^2=\mbox{trace}(z^*z)$.
$\langle\cdot,\cdot\rangle$ denotes the usual scalar product in
$\R^m$.

Let $(\Omega,{\cal F},P)$ be a complete probability space. By $W$
we denote  a standard $d$-dimensional Wiener process on $(\Omega,
{\cal F}, P)$  and by $(\FF_t)$ the standard augmentation of the
natural filtration generated by $W$.

$L^p$, $p\geq1$, is the space of random vectors $X$ such that
$\|X\|_p=E(|X|^p)^{1/p}<\infty$. ${\SSS}^p$ is the space of
c\`adl\`ag adapted (with respect to $(\FF_t)$) processes $X$ such
that $\|X\|_{{\SSS}^p} =\|X^*_T\|_p<\infty$ and ${\PP}^p$ is the
space of progressively measurable $m\times d$-dimensional
processes $Z$ such that $\|Z\|_{{\PP}^p}
=\|(\int_0^T\|Z_s\|^2\,ds)^{1/2}\|_p<\infty$. ${\SSS}$ is the
space of c\`adl\`ag  adapted processes equipped with the metric
$\delta(X,X')=E((X-X')^*\wedge1)$  and ${\cal P}$ is the space of
progressively measurable $m\times d$-dimensional processes $Z$
such that $\int_0^T\|Z_s\|^2\,ds<\infty$, $P$-a.s. equipped with
the metric $\delta'(Z,Z')=E(\int_0^T\|Z_s-Z'_s\|^2\,ds\wedge1)$.
It is well known that ${\SSS}^p$, ${\PP}^p$ are Banach spaces for
$p\geq1$ and that ${\SSS}$, ${\PP}$  are complete metric spaces.
By ${\HH}^2$ we denote the space of $m$-dimensional special
semimartingales equipped with the norm
\[
\|X\|_{{\HH}^2} =\|[M]^{1/2}_T\|_{L^2} +\||B|_T\|_{L^2},
\]
where $X=M+B$ is the canonical decomposition of $X$, $[M]_T$
is the quadratic variation of $M$ at $T$ and $|B|_T$ is the variation of
$B$ on the interval $[0,T]$.

Given a process $Y$ and an $(\FF_t)$-stoping time $\tau$ we denote
by $Y^{\tau}$ the stopped process $\{Y_{t\wedge\tau};t\in[0,T]\}$.

By $Conv$ we denote the space of all bounded closed convex subsets
of $\R^m$ with nonempty interiors endowed with the Hausdorff
metric $\rho$, i.e. any $G,G'\in Conv$,
\[
\rho(G,G')=\max\big(\sup_{x\in G}\dist(x,G'),\sup_{x\in G'}
\dist(x,G)\big),
\]
where $\dist(x,G)=\inf_{y\in G}|x-y|$).

\begin{remark}[see Protter \cite{Pr}]
\label{rem2.1}
{(a)} For a special semimartingale $X$,
\[
\|X\|_{{\cal H}^2}\leq 3 \sup_{H\,\mbox{\rm\tiny
predictable},\,|H|\leq1} \|(\int_0^{\cdot}
H_s\,dX_s)^*\|_{L^2}\leq9\|X\|_{{\cal H}^2}.
\]
(b) $\|X\|_{{\cal S}^2}\leq c\|X\|_{{\cal H}^2}$  and
$\|[X]_T^{1/2}\|_{\LL^2}\leq\|X\|_{{\cal H}^2}$. Moreover, for any
predictable and locally bounded $H$,
\[
\|\int_0^{\cdot} H_s\,dX_s\|_{{\cal H}^2}
\leq\|H^*\|_{L^2}\|X\|_{{\cal H}^2}.
\]
\end{remark}

\begin{remark}[see Menaldi \cite{Me}]
\label{rem2.2} (a) Let $G$ be a closed convex domain with nonempty
interior and let ${\cal N}_{y}$ denote the set of inward normal
unit vectors at $y\in\partial G$. It is well known that ${\bf
n}\in{\cal N}_{y}$ iff $\langle y-x,\mbox{\bf n}\rangle\leq 0$.
for every $x\in G$ (Here $\langle\cdot,\cdot\rangle$ stands for the usual
inner product in $\R^d$).\\
(b) If moreover $a\in\mbox{\rm Int} G$ then for every  ${\bf n}
\in{\cal N}_{y}$,
\[
\langle y-a,\mbox{\bf n}\rangle\leq-\dist(a,\partial G).
\]
(c) If $\dist(x,G)>0$ then there exists a unique
$y=\Pi_G(x)\in\partial G$ such that $|y-x|=\dist(x,G)$. One can
observe that $(y-x)/|y-x|\in{\cal N}_{y}$. Moreover, for every
$a\in\mbox{\rm Int} G$,
\[
\langle x-a,y-x\rangle\leq-\dist(a,\partial G)|y-x|.
\]
(d) For all $x,x'\in\R^m$,
\[
\langle x-x',(x-\Pi_G(x))-(x'-\Pi_G(x'))\rangle\geq0.
\]
\end{remark}

In the paper we  assume that we are given an ${\cal
F}_T$-measurable $m$-dimensional random vector $\xi$, a generator
$f:[0,T]\times\Omega\times\R^m\times\R^{m\times d}\to\R^m$, which
is measurable with respect to $Prog\otimes {\cal B}
(\R^m)\otimes{\cal B}(\R^{m\times d})$, where $Prog$ denotes the
$\sigma$-field of all progressive subsets of $[0,T]\times\Omega$
and a family $\DD=\{D_t;\,t\in[0,T]\}$ of random  closed convex
sets in $\R^m$ with nonempty interiors such that the process
$[0,T]\ni t\mapsto D_t\in Conv$ is $(\FF_t)$-adapted. Moreover, we
will assume that
\begin{enumerate}
\item[(H1)]$\xi\in D_T$, $\xi\in L^2$,
\item[(H2)]$E\int_0^T|f(s,0,0)|^2\,ds<\infty$,
\item[(H3)]There are $\mu,\lambda\ge0$ such that for any $t\in[0,T]$,
\[
|f(t,y,z)-f(t,y',z')|\le\mu|y-y'|+\lambda\|z-z'\|,\quad
\,y,y'\in\R^m,\, z,z'\in\R^{m\times d},
\]
\item[(H4)]For each $N\in\N$ the mapping $t\to D_t\cap B(0,N)\in Conv$
is c\`adl\`ag $P$-a.s. (with the convention that $D_{T}=D_{T-}$)
and there is a semimartingale $A\in{\cal H}^2$ such that
$A_t\in\mbox{\rm Int} D_t$ for $t\in[0,T]$  and
\[
\inf_{t\leq T}\mbox{\rm dist}(A_t,\partial D_t)>0.
\]
\end{enumerate}

In (H1), (H3), (H4) and  in the sequel we understand that the
equalities and inequalities hold true $P$-a.s.

\begin{definition}
\label{def2.3} We say that  a triple $(Y, Z,K)$  of $(\FF_t)$-progressively
measurable  processes is a solution of the RBSDE (\ref{eq1.1}) if
\begin{enumerate}
\item[(a)]$Y_t\in D_t$, $t\in[0,T]$,
\item[(b)] $K$ is a c\`adl\`ag process of locally bounded  variation
such that $K_0=0$ and for every $(\FF_t)$ adapted c\`adl\`ag
process $X$ such that $X_t\in D_t$, $t\in[0,T]$, we have
\[
\int_0^T\langle Y_{s-}-X_{s-},dK_s\rangle\leq 0,
\]
\item[(c)]Eq. (\ref{eq1.1}) is satisfied.
\end{enumerate}
\end{definition}

\begin{proposition}
\label{prop2.4} Assume \mbox{\rm(H1)--(H4)}. If $(Y,Z,K)$ is a
solution of \mbox{\rm(\ref{eq1.1})} such that $Y\in{\cal S}^2$
then there exists $C>0$ depending only on $\mu,\lambda, T$ such
that
\begin{align*}
&E\big((Y^*_T)^2+\int_0^T\|Z_s\|^2\,ds
+\sum_{0<s\leq T}|\Delta K_s|^2+(K^*_T)^2
+\int_0^T\dist(A_{s-},\partial D_{s-})\,d|K|_s\big) \\
&\qquad\leq C\Big(
E\big(|\xi|^2+\int_0^T|f(s,0,0)|^2\,ds\big)+\|A\|^2_{{\cal
H}^2}\Big)
\end{align*}
\end{proposition}
\begin{proof} We first show that
\begin{align}
\label{eq2.1} &\nonumber E\big(\int_0^T\|Z_s\|^2\,ds
+\sum_{0<s\leq T}|\Delta K_s|^2+(K^*_T)^2
+\int_0^{T}\dist(A_{s-},\partial D_{s-})\,d|K|_s\big)\big)\\
&\qquad\leq
C\Big(E\big((Y^*_T)^2+\int_0^T|f(s,0,0)|^2\,ds\big)
+\|A\|^2_{{\cal H}^2}\Big).
\end{align}
Let $\tau_n=\inf\{t>0;\int_0^t\|Z_s\|^2ds>n\}\wedge T$, $n\in\N$.
By It\^o's formula,
\begin{align*}
|Y_0|^2+\int_0^{\tau_n}\|Z_s\|^2\,ds+\sum_{s\leq \tau_n}|\Delta
K_s|^2
&=|Y_{\tau_n}|^2+2\int_0^{\tau_n}\langle Y_{s},f(s,Y_s,Z_s)\rangle\,ds\\
&\quad-2\int_0^{\tau_n}\langle Y_{s},Z_s\,dW_s\rangle
+2\int_0^{\tau_n}\langle Y_{s-},dK_s\rangle.
\end{align*}
By Remark \ref{rem2.2}(b), the integration by parts formula and
the fact that $dK_t=\mbox{\bf n}_{Y_t}\,d|K|_t$,
\begin{align}
\label{eq2.02} \int_0^{\tau_n}\langle Y_{s-},dK_s\rangle&
=\int_0^{\tau_n}\langle Y_{s-}-A_{s-},dK_s\rangle
+\int_0^{\tau_n}\langle A_{s-},dK_s\rangle\nonumber\\
&\leq-\int_0^{\tau_n}\dist(A_{s-},\partial D_{s-})\,d|K|_s\nonumber\\
&\quad+\langle K_{\tau_n},A_{\tau_n}\rangle
-\int_0^{\tau_n}\langle K_{s-},dA_s\rangle -\sum_{s\leq
\tau_n}\langle \Delta K_s,\Delta A_s\rangle,
\end{align}
whereas by (H3),
\[
\int_0^{\tau_n}\langle Y_{s},f(s,Y_s,Z_s)\rangle\,ds \leq
C_1\big((Y^*_T)^2+\int_0^T|f(s,0,0)|^2ds\big)
+\frac14\int_0^{\tau_n}||Z_s||^2ds
\]
for some $C_1>0$. Putting together the above inequalities and
using the fact that
\[
-2\sum_{s\leq \tau_n}\langle \Delta K_s,\Delta A_s\rangle
\leq2(\sum_{s\leq \tau_n}|\Delta K_s|^2)^{1/2}(\sum_{s
\leq \tau_n}|\Delta A_s|^2)^{1/2}
\leq\frac12{\sum_{s\leq \tau_n}|\Delta K_s|^2}+2[A]_T
\]
we obtain
\begin{align}
&\nonumber\int_0^{\tau_n}\|Z_s\|^2\,ds
+\sum_{s\leq \tau_n}|\Delta K_s|^2
+4\int_0^{\tau_n}\dist(A_{s-},\partial D_{s-})\,d|K|_s\\
&\qquad \nonumber \leq 2(Y^*)^2+2C_1\big((Y^*_T)^2
+\int_0^T|f(s,0,0)|^2ds\big)+4K^*_{\tau_n}A^*_T+4[A]_T\\
&\qquad\quad-4\int_0^{\tau_n}\langle
K_{s-},dA_s\rangle-4\int_0^{\tau_n}\langle
Y_{s},Z_s\,dW_s\rangle\label{eq2.2}.
\end{align}
Since $K_t=Y_0-Y_t-\int_0^tf(s,Y_s,Z_s)\,ds+\int_0^t Z_s\,dW_s$,
$t\in[0,T]$, we have
\[
K^*_{\tau_n}\leq(2+\mu T)Y^*_T
+\int_0^T|f(s,0,0)|\,ds+\lambda\int_0^{\tau_n}\|Z_s\|\,ds
+\sup_{t\leq \tau_n}|\int_0^t Z_s\,dW_s|.
\]
Therefore there is $C_2>0$ such that
\begin{equation}
\label{eq2.3} E(K^*_{\tau_n})^2\leq C_2E\big((Y^*_{T})^2
+\int_0^T|f(s,0,0)|^2\,ds+\int_0^{\tau_n}\|Z_s\|^2\,ds\big).
\end{equation}
Since $E|\int_0^{\tau_n}\langle K_{s-},dA_s\rangle|\leq
c\|K^*_{\tau_n}\|_{L^2}\|A\|_{{\cal H}^2}$ and
$\int_0^{t\wedge\tau_n}\langle Y_{s},Z_s\,dW_s\rangle$  is a
uniformly integrable martingale, from (\ref{eq2.2}) it follows
that there is $C_3>0$ such that
\begin{align}
&\nonumber E\big(\int_0^{\tau_n}\|Z_s\|^2\,ds +\sum_{s\leq
\tau_n}|\Delta K_s|^2
+4\int_0^{\tau_n}\dist(A_{s-},\partial D_{s-})\,d|K|_s\big)\\
&\qquad \label{eq2.4}\leq
C_3\Big(E\big((Y^*)^2+\int_0^T|f(s,0,0)|^2\,ds
+(A^*_T)^2+[A]_T\big)+\|A\|^2_{{\cal H}^2}\Big)\nonumber\\
&\qquad\quad+(2C_2)^{-1}E(K^*_{\tau_n})^2.
\end{align}
Combining (\ref{eq2.3}) with (\ref{eq2.4}), using the fact that
$E(A^*_T)^2, E([A]_T)\leq \|A\|^2_{{\cal H}^2}$ and then letting $n\to
\infty$ we obtain (\ref{eq2.1}).

In the second part of the proof  we will estimate $E(Y^*_T)^2$.
Using It\^o's formula gives
\begin{align}
&|Y_t|^2+\int_t^{T}\|Z_s\|^2\,ds+\sum_{t<s\leq T}|\Delta K_s|^2\nonumber\\
&\quad=|\xi|^2+2\int_t^{T}\langle Y_{s},f(s,Y_s,Z_s)\rangle\,ds
-2\int_t^{T}\langle Y_{s},Z_s\,dW_s\rangle
+2\int_t^{T}\langle Y_{s-},dK_s\rangle\nonumber\\
&\quad\leq|\xi|^2+\int_0^T|f(s,0,0)|^2\,ds+(2\mu
+2\lambda^2+1)\int_t^T|Y_s|^2\,ds\nonumber\\
&\qquad+\frac12\int_t^T\|Z_s\|^2\,ds-2\int_t^{T}\langle
Y_{s},Z_s\,dW_s\rangle+2\int_t^{T}\langle A_{s-},dK_s\rangle.
\end{align}
From this we deduce that there is $C_4>0$ such that
\begin{equation}
\label{eq2.5} |Y_t|^2+\frac12\int_t^{T}\|Z_s\|^2\,ds\leq
X+C_4\int_t^T|Y_s|^2\,ds-2\int_t^{T}\langle
Y_{s},Z_s\,dW_s\rangle,\quad t\in[0,T],
\end{equation}
where $X=|\xi|^2+\int_0^T|f(s,0,0)|^2\,ds +\sup_{t\leq T}|\int_0^t
\langle A_{s-},dK_s\rangle$. Note that from (H1), (H2) and earlier
considerations it follows that $X$ is integrable. Since
$\int_0^t\langle Y_{s},Z_s\,dW_s\rangle$ is a uniformly integrable
martingale,
\[
E\int_t^{T}\|Z_s\|^2\,ds\leq 2EX+2C_4E\int_t^T|Y_s|^2\,ds,\quad
t\in[0,T].
\]
Consequently,
\begin{align}
\label{eq2.6} \nonumber E\sup_{s\in[t,T]}|Y_s|^2&\leq
EX+C_4E\int_t^T|Y_s|^2\,ds+2cE(\int_t^T|Y_s|^2\|Z_s\|^2\,ds)^{1/2}\\
&\leq\nonumber EX+C_4E\int_t^T|Y_s|^2\,ds
+2cE\big(\sup_{s\in[t,T]}|Y_s|(\int_t^T|Y_s|^2\|Z_s\|^2\,ds)^{1/2}\big)\\
\nonumber &\leq EX+C_4E\int_t^T|Y_s|^2\,ds
+\frac12E(\sup_{s\in[t,T]}|Y_s|^2)+2c^2E(\int_t^T\|Z_s\|^2\,ds)^{1/2}\\
&\leq(4c^2+1)EX+C_4(4c^2+1)E\int_t^T|Y_s|^2\,ds
+\frac12E(\sup_{s\in[t,T]}|Y_s|^2).
\end{align}
Therefore there are  $C_5,C_6>0$ such that
\[
E\sup_{s\in[t,T]}|Y_s|^2\leq C_5EX
+C_6\int_t^TE(\sup_{u\in[s,T]}|Y_u|^2)\,ds,\quad t\in[0,T].
\]
Hence, by Gronwall's lemma,  $E(Y^*_T)^2\leq C_5EXe^{C_6T}$. Since
by the integration by part formula and the previously used
arguments there is $C_7>0$ such that
\[
E\sup_{t\leq T}|\int_0^t \langle A_{s-},dK_s\rangle \leq
(2C_1C_5)^{-1}e^{-C_6T}E\big(\sum_{0<s\leq T}|\Delta K_s|^2
+(K^*_T)^2\big)+C_7\|A\|^2_{{\cal H}^2},
\]
it follows from (\ref{eq2.1}) that
\begin{align*}
E(Y^*_T)^2&\leq C_5e^{C_6T}E\big(|\xi|^2+\int_0^T|f(s,0,0)|^2\,ds\big)
+\frac12E(Y^*_T)^2\\
&\quad+\frac12\big(E\int_0^T|f(s,0,0)|^2\,ds+\|A\|^2_{{\cal H}^2}\big)
+C_5C_7e^{C_6T}\|A\|^2_{{\cal H}^2},
\end{align*}
which completes the proof.
\end{proof}

Let $\DD'=\{D'_t;\,t\in[0,T]\}$ be another family of random closed
sets satisfying (H4) with some semimartingale $A'$ and let $\xi'$
be an $\FF_T$-measurable  random variable such that $\xi'\in D'_T$
and $\xi'\in L^2$. In the following proposition together with
(\ref{eq1.1}) we consider RBSDE with terminal condition $\xi'$,
coefficient $f$ and family $\DD'$, i.e. equation of the form
\begin{equation}
\label{eq2.7}
Y'_t=\xi'+\int^T_tf(s,Y'_s,Z'_s)\,ds
-\int^T_t\langle Z'_s\,dW_s\rangle+K'_T-K'_t,\quad t\in[0,T].
\end{equation}
In its proof we will use the following notation
\[
\sgn(x)=\frac{x}{|x|}{\bf 1}_{\{x\neq0\}},\quad x\in\Rd.
\]

\begin{proposition}
\label{prop2.5} Let $(Y,Z,K)$ and $(Y',Z',K')$ be solutions of
\mbox{\rm(\ref{eq1.1})} and \mbox{\rm(\ref{eq2.7})}, respectively,
and let $\bar Y=Y-Y'$, $\bar Z=Z-Z'$, $\bar K=K-K'$. If $f$
satisfies \mbox{\rm(H3)} and $Y,Y'\in{\cal S}^2$  then for every
$p\in(1,2]$ there exists $C>0$ depending only on $\mu,\lambda, T$
such that for any stopping time $\sigma$ such that
$0\leq\sigma\leq T$ we have
\begin{align*}
&E\big(( \sup_{t<\sigma}|\bar Y_t|^p
+\int_0^\sigma|\bar Y_s|^{p-2}{\bf 1}_{\{\bar Y_s\neq0\}}\|\bar Z_s\|^2\,ds
+I_{\sigma-}\big)\\
&\quad\leq C\Big(E(|\bar Y_{\sigma-}|^p +\int_0^{\sigma-}|\bar
Y_{s-}|^{p-2}|
\Pi_{D_{s-}}(Y'_{s-})-Y'_{s-}|{\bf 1}_{\{Y'_{s-}\notin D_{s-}\}}\,d|K|_s)\\
&\qquad\qquad+E(\int_0^{\sigma-}|\bar
Y_{s-}|^{p-2}|\Pi_{D'_{s-}}(Y_{s-})-Y_{s-}|{\bf 1}_{\{Y_{s-}\notin
D'_{s-}\}}\,d|K'|_s\Big),
\end{align*}
where $I_t=\sum_{s\leq t}(|\bar Y_s|^p-|\bar Y_{s-}|^p-p|\bar
Y_{s-}|^{p-1}\langle \sgn(\bar Y_{s-}),\Delta \bar Y_s\rangle)$,
$t\ge0$.
\end{proposition}
\begin{proof}
By It\^o's formula for the convex function $x\to|x|^p$  (see
\cite{Kl:BSM}), for any $t<\sigma$ we have
\begin{align*}
&|\bar Y_t|^p+\frac{p(p-1)}{2}\int_t^\sigma|\bar Y_s|^{p-2}{\bf
1}_{\{\bar Y_s\ne 0\}}\|\bar Z_s\|^2\,ds+I_{\sigma-}-I_t\\
&\quad=|\bar Y_{\sigma-}|^p
+p\int_t^{\sigma}|\bar Y_{s}|^{p-1}
\langle \sgn(\bar Y_s),f(s,Y_s,Z_s-f(s,Y'_s,Z'_s)\rangle\,ds\\
&\qquad-p\int_t^{\sigma}|\bar Y_{s}|^{p-1}\langle \sgn(\bar
Y_s),\bar Z_s\,dW_s\rangle +p\int_t^{\sigma-}| \bar
Y_{s-}|^{p-1}\langle\sgn(\bar Y_{s-}),d\bar K_s\rangle.
\end{align*}
Since
\begin{align*}
&p\int_t^{\sigma}|\bar Y_{s}|^{p-1}
\langle \sgn(\bar Y_s),f(s,Y_s,Z_s-f(s,Y'_s,Z'_s)\rangle\,ds\\
&\quad\leq C_1 \int_t^{\sigma}|\bar Y_{s-}|^p\,ds
+\frac{p(p-1)}4\int_t^{\sigma}|\bar  Y_{s-}|^{p-2}{\bf 1}_{\{\bar
Y_s\ne 0\}}\|\bar Z_s\|^2\,ds
\end{align*}
for some $C_1>0$, it follows that
\begin{align}
\label{eq2.09} &|\bar Y^{\sigma-}_t|^p
+\frac{p(p-1)}{4}\int_t^{t\wedge\sigma}|\bar Y_s|^{p-2}{\bf
1}_{\{\bar Y_s\ne 0\}}\|\bar Z_s\|^2\,ds+I_{t\wedge\sigma-}-I_t
\nonumber\\
&\quad\le|\bar Y_{\sigma-}|^p+ C_1p \int_t^{T}|\bar
Y^{\sigma-}_{s-}|^p\,ds -p\int_t^{t\wedge\sigma}|\bar
Y_{s}|^{p-1}\langle \sgn(\bar Y_s),\bar Z_s\,dW_s\rangle\nonumber\\
&\qquad+p\int_t^{t\wedge\sigma-}| \bar
Y_{s-}|^{p-1}\langle\sgn(\bar Y_{s-}),d\bar K_s\rangle
\end{align}
for $t\in[0,T]$. Since $ \int_t^{\sigma}| \bar
Y_{s-}|^{p-1}\langle\sgn(\bar Y_{s-}),d\bar
K_s\rangle=\int_t^{\sigma}| \bar Y_{s-}|^{p-2}{\bf 1}_{\{Y_s\ne
Y'_s\}}\langle \bar Y_{s-},d\bar K_s\rangle$ and
\begin{align}
\label{eq2.8} \nonumber\langle \bar Y_{s-},d\bar K_s\rangle
&=\langle Y_{s-}-\Pi_{D_{s-}}(Y'_{s-}),dK_s\rangle+
\langle Y'_{s-}-\Pi_{D'_{s-}}(Y_{s-}),dK'_s\rangle\\
&\qquad\nonumber+\langle \Pi_{D_{s-}}(Y'_{s-})-Y'_{s-},dK_s\rangle
+\langle \Pi_{D'_{s-}}(Y_{s-})-Y_{s-},dK'_s\rangle\\
&\leq|\Pi_{D_{s-}}(Y'_{s-})-Y'_{s-}|\,d|K|_s
+|\Pi_{D'_{s-}}(Y_{s-})-Y_{s-}|\,d|K'|_s,
\end{align}
we see that $\sup_{t\in[0,T]}\int_t^{t\wedge\sigma-}| \bar
Y_{s-}|^{p-1}\langle\sgn(\bar Y_{s-},d\bar K_s\rangle\leq X_1$,
where
\begin{align*}
X_1&= \int_0^{\sigma-}|\bar Y_{s-}|^{p-2}
|\Pi_{D_{s-}}(Y'_{s-})-Y'_{s-}|{\bf 1}_{\{Y'_{s-}\notin D_{s-}\}}\,d|K|_s\\
&\quad+\int_0^{\sigma-}|\bar
Y_{s-}|^{p-2}|\Pi_{D'_{s-}}(Y_{s-})-Y_{s-}|{\bf 1}_{\{Y_{s-}\notin
D'_{s-}\}}\,d|K'|_s.
\end{align*}
From (\ref{eq2.09}), (\ref{eq2.8}) and the fact that
$\int_0^{t}|\bar Y_{s}|^{p-1}\langle \sgn(\bar Y_s),\bar
Z_s\,dW_s\rangle$ is a uniformly integrable martingale it follows
that
\begin{align}
\label{eq2.9}
\nonumber&\frac{p(p-1)}{4}E\int_t^{t\wedge\sigma}|\bar
Y_s|^{p-2}{\bf 1}_{\{\bar Y_s\ne 0\}}\|\bar Z_s\|^2\,ds
+E(I_{t\wedge\sigma-}-I_t)\\
&\quad\qquad\leq EX+Ep C_1 \int_t^{T}|\bar
Y^{\sigma-}_{s-}|^p\,ds, \quad t\in[0,T],
\end{align}
where $X=|\bar Y_{\sigma-}|^p+X_1$. Arguing as in the proof of
(\ref{eq2.6}) we deduce from the above that there exist constants
$C_5,C_6>0$ such that
\[
E\sup_{s\in[t,T]}|\bar Y_s^{\sigma-}|^p \leq
C_5EX+C_6\int_t^TE(\sup_{u\in[s,T]}|\bar Y^{\sigma-}_u|^p)\,ds,
\quad t\in[0,T].
\]
By Gronwall's lemma, $E(\sup_{t<\sigma}|\bar Y_t|^p)
=E\sup_{t\in[0,T]}|\bar Y_t^{\sigma-}|^p\leq C_5EXe^{C_6T}$.
Putting $t=0$ in (\ref{eq2.9})  completes the proof.
\end{proof}

\begin{corollary}
\label{cor2.6} Under the assumptions of Proposition \ref{prop2.5},
if moreover $D^{\sigma-}=D'^{\sigma-}$ then
\[
E\big(( \sup_{t<\sigma}|\bar Y_t|^p +\int_0^\sigma|\bar
Y_s|^{p-2}{\bf 1}_{\{\bar Y_s\neq0\}}\|\bar Z_s\|^2\,ds
+I_{\sigma-}\big)\leq CE(|\bar Y_{\sigma-}|^p).
\]
\end{corollary}

\begin{remark}
\label{rem2.7}
Since $\Delta \bar Y=-\Delta \bar K$, in the case  $p=2$ we have
\[
I_t=\sum_{s\leq t}|\Delta \bar K_s|^2,\quad t\ge0.
\]
\end{remark}

\nsubsection{Existence and uniqueness of solutions of RBSDEs}
\label{sec3}

Our main goal is to prove that under (H1)--(H4) there exists a
unique solution $(Y,Z,K)$  of (\ref{eq1.1}) such that $Y,K\in{\cal
S}^2$ and  $Z\in{\cal P}^2$. The uniqueness follows easily from
Corollary \ref{cor2.6}. In the proof of the existence we will use
the method of approximation of $\DD$  by discrete time-dependent
process described in the following proposition.

\begin{proposition}
\label{prop3.2} Let
$\sigma_0=0\leq\sigma_1\leq\ldots\leq\sigma_{k+1}=T$ be  stopping
times and let $D^0,D^1,\dots,D^{k}$  be random closed convex
subsets of $\R^m$ with nonempty interiors such that $D^i$ is
$\FF_{\sigma_i}$-measurable. Let $(Y,Z,K)$ be a triple of
$(\FF_t)$-progressively measurable processes such that
\begin{enumerate}
\item[\rm(a)]
$\xi=Y_T\in D^{k}$, $K$ is a c\`adl\`ag process of locally bounded
variation such that $K_0=0$, $\Delta
K_{\sigma_i}=\Pi_{D^{i-1}}(Y_{\sigma_i})-Y_{\sigma_i}=-\Delta
Y_{\sigma_i}$,
\item[\rm(b)]on each interval $[\sigma_{i-1},\sigma_i)$,
$i=1,\dots,{k+1}$, we have
\begin{equation}
\label{eq3.1}
Y_t=\Pi_{D^{i-1}}(Y_{\sigma_i})+\int^{\sigma_i}_tf(s,Y_s,Z_s)\,ds
-\int^{\sigma_i}_tZ_s\,dW_s +K_{\sigma_i-}-K_t,
\end{equation}
where $Y_t\in D^{i-1}$,
\item[\rm(c)]$\int_{(\sigma_{i-1}, \sigma_{i})}\langle
Y_{s-}-X_{s-}\,dK_s\rangle \leq0$ for every $(\FF_t)$-adapted c\`adl\`ag
process $X$ such that $X_t\in D^{i-1}$ for
$t\in[\sigma_{i-1},\sigma_i)$.
\end{enumerate}
Then $(Y,Z,K)$ is a unique solution of \mbox{\rm(\ref{eq1.1})}
with terminal value $\xi$ and $\{D_t;\,t\in[0,T]\}$ such that
$D_t=D^{i-1}$, $t\in[\sigma_{i-1},\sigma_i[$, $i=1,\dots,{k+1}$,
$D_T=D_{T-}$\,.
\end{proposition}
\begin{proof}
Note that
$Y_{\sigma_i}+K_{\sigma_i}=\Pi_{D^{i-1}}(Y_{\sigma_i})+K_{\sigma_i-}$,
$i=1,\dots,k$, so  $(Y, Z,K)$ satisfies (\ref{eq1.1}). Since
$Y_t\in D_t$, we only have to check condition (b) of the
definition of a solution of (\ref{eq1.1}). Let $X$ be an
$(\FF_t)$-adapted c\`adl\`ag process such that  $X_t\in D_t$,
$t\in[0,T]$. Clearly
\[
\sum_{i=1}^{k+1}\int_{(\sigma_{i-1},\sigma_i)}\langle
Y_{s-}-X_{s-}\,,dK_s\rangle \leq0.
\]
On the other hand,
\[
Y_{\sigma_i-}=Y_{\sigma_i}-\Delta Y_{\sigma_i}
=Y_{\sigma_i}+\Delta
K_{\sigma_i}=\Pi_{D_{i-1}}(Y_{\sigma_i}),\quad i=1,\dots,k.
\]
Since $X_{\sigma_i-}\in D_{\sigma_i-}$\,, from  Remark 2.1(a) it
follows that
\begin{align*}
\langle Y_{\sigma_i-}-X_{\sigma_i-},\Delta K_{\sigma_i}\rangle
&=\langle  Y_{\sigma_i-}-X_{\sigma_i-},\Pi_{D_{\sigma_i-}}(Y_{\sigma_i})
-Y_{\sigma_i}\rangle\\
&=\langle
\Pi_{D_{\sigma_i-}}(Y_{\sigma_i})-X_{\sigma_i-},
\Pi_{D_{\sigma_i-}}(Y_{\sigma_i})-Y_{\sigma_i}\rangle\leq0,
\end{align*}
which completes the proof.
\end{proof}
\medskip

Now we are going to study the problem of existence of solutions of
(\ref{eq3.1}). To this end, we first consider local RBSDEs on
closed random intervals.

Let $\tau,\sigma$ be stopping times such that
$0\leq\tau\leq\sigma\leq T$, $D$ be an $\FF_\tau$-measurable
random convex set with nonempty interior and let $\zeta\in L^2$ be
an $\FF_\sigma$-measurable random variable.  We consider equations
of the form
\begin{equation}
\label{eq3.2} Y_t=\zeta+\int^{\sigma}_tf(s,Y_s,Z_s)\,ds
-\int^{\sigma}_tZ_s\,dW_s +K_{\sigma}-K_t,\quad t\in[\tau,\sigma].
\end{equation}

\begin{definition}
\label{def3.2}
We say that  a triple $(Y, Z,K-K_\tau)$  of
$(\FF_t)$-progressively measurable  processes on $[\tau,\sigma]$
is a solution of the local RBSDE  on $[\tau,\sigma]$ if it satisfies
(\ref{eq3.2})  and
\begin{enumerate}
\item[(a)] $Y_t\in D$, $t\in[\tau,\sigma]$,
\item[(b)]$K_{\tau}=0$, $K$ is a c\`adl\`ag process of locally bounded
variation on the interval $[\tau,\sigma]$ such that
$\int_\tau^\sigma\langle Y_{s-}-X_{s-},dK_s\rangle\leq 0$ for
every $(\FF_t)$-adapted c\`adl\`ag process $X$ with values in $D$.
\end{enumerate}
\end{definition}

We will assume that
\begin{enumerate}
\item[(H1${}^*$)]$\zeta\in D$, $\zeta\in L^2$,
\item[(H2${}^*$)]$E\int_\tau^\sigma|f(s,0,0)|^2\,ds<\infty$,
\item[(H3${}^*$)]There are $\mu,\lambda\ge0$ such that for any
$t\in[\tau,\sigma]$,
\[
|f(t,y,z)-f(t,y,z')|\le\mu|y-y'|+\lambda\|z-z'\|,\quad
\,y,y'\in\R^m,\, z,z'\in\R^{m\times d},
\]
\item[(H4${}^*$)]There is an $\FF_\tau$-measurable random variable
$A\in L^2$ such that $A\in\mbox{\rm Int}D$.
\end{enumerate}

\begin{proposition}
\label{cor3.4} Assume \mbox{\rm(H1${}^*$)--(H4${}^*$)}. If
$(Y,Z,K)$ is a solution of \mbox{\rm(\ref{eq3.2})} such that
$\sup_{\tau\leq t\leq\sigma}|Y_t|\in L^2$ then there exists $C>0$
depending only on $\mu,\lambda, T$ such that
\[
E\big(\sup_{\tau\leq t\leq\sigma}|Y_t|^2
+\int_\tau^{\sigma}\|Z_s\|^2\,ds\,|\,\FF_{\tau}\big) \le
CE\Big(|\zeta|^2+|A|^2
+\int_\tau^\sigma|f(s,0,0)|^2\,ds\,|\,\FF_{\tau}\Big)
\]
and
\[
E\big(|K|_\tau^\sigma\,|\,\FF_\tau\big) \le  C(\dist(A,\partial
D))^{-1}E\Big(|\zeta|^2+|A|^2
+\int_\tau^{\sigma}|f(s,0,0)|^2\,ds\,|\,\FF_{\tau}\Big).
\]
\end{proposition}
\begin{proof}
It is sufficient to apply arguments from the proof of Proposition
\ref{prop2.4} and use the fact that  $\dist(A,\partial D)$  is a
strictly positive  $\FF_\tau$-measurable random variable.
\end{proof}
\medskip

Let $D'$ be an $\FF_\tau$-measurable random convex set with
nonempty interior, $\zeta'\in L^2$ be an $\FF_\sigma$-measurable
random variable such that $\zeta'\in D'$ $P$-a.s. and there is an
$\FF_\tau$-measurable random variable $A'\in L^2$ such that
$A'\in\mbox{\rm Int} D'$. Consider the local RBSDE on
$[\tau,\sigma]$ of the form
\begin{equation}
\label{eq3.3} Y'_t=\zeta'+\int^{\sigma}_tf(s,Y'_s,Z'_s)\,ds
-\int^{\sigma}_tZ'_s\,dW_s +K'_{\sigma}-K'_t,\quad
t\in[\tau,\sigma].
\end{equation}

\begin{proposition}
\label{cor3.5} Let $(Y,Z,K)$ and $(Y',Z',K')$ be solutions of
\mbox{\rm(\ref{eq3.2})} and \mbox{\rm(\ref{eq3.3})}, respectively.
If $f$ satisfies \mbox{\rm(H3${}^*$)} and $\sup_{\tau\leq
t\leq\sigma}|Y_t|$, $\sup_{\tau\leq t\leq\sigma}|Y'_t| \in L^2$
then there exists $C>0$ depending only on $\mu,\lambda,T$ such
that
\begin{align*}
&E\big(\sup_{\tau\leq t\leq\sigma}|Y_t-Y'_t|^2
+\int_\tau^\sigma\|Z_s-Z'_s\|^2\,ds\,|\,\FF_{\tau}\big)\\
&\qquad\le C\Big(E(|\zeta-\zeta'|^2\,|\,\FF_{\tau})
+E(\sup_{\tau<t\leq\sigma}|\Pi_D
(Y'_{t-})-Y'_{t-}||K|_\tau^{\sigma}\,|\,\FF_{\tau})\\
&\qquad\qquad\quad+E(\sup_{\tau< t\leq\sigma}
|\Pi_{D'}(Y_{t-})-Y_{t-}||K'|_\tau^{\sigma}\,|\,\FF_{\tau})\Big).
\end{align*}
\end{proposition}
\begin{proof}
We apply the arguments from the proof of Proposition \ref{prop2.5}
with $p=2$.
\end{proof}
\medskip

We will need the following assumption: there exists $N\in\N$ such
that
\begin{equation}
\label{eq3.4} D\subset B(0,N).
\end{equation}

\begin{proposition}
\label{prop3.6} Assume \mbox{\rm(H1${}^*$)--(H4${}^*$)} and
\mbox{\rm(\ref{eq3.4})}. Then there exists a unique solution
$(Y,Z,K)$ of the local RBSDE \mbox{\rm(\ref{eq3.2})} such that
\begin{equation}
\label{eq3.05}
\sup_{\tau\leq t\leq \sigma}|Y_t|\in
L^2,\quad\int_\tau^\sigma\|Z_s\|^2\,ds\in L^1, \quad\sup_{\tau\leq
t\leq \sigma} |K_t-K_{\tau}|\in L^2.
\end{equation}
\end{proposition}
\begin{proof}
The uniqueness follows from Proposition \ref{cor3.5}. To prove the
existence we first assume that $D$ is nonrandom, i.e. $D=G$, where
$G$ is some fixed convex set with nonempty interior. Set
$g(s,\cdot,\cdot)=f(s,\cdot,\cdot){\bf 1}_{[0,\sigma[}(s)$. By
\cite[Theorem 5.9]{GP} there exists a solution $(Y,Z,K)$ of the
following RBSDE in $G$
\begin{equation}
\label{eq3d1} Y_t=\zeta+\int_t^Tg(s,Y_s,Z_s)\,ds
-\int_t^TZ_s\,dW_s+K_T-K_t,\quad t\in[0,T].
\end{equation}
Since  $Y_t=\zeta$, $Z_t=0$  and $K_T=K_t$ for $t\geq\sigma$, it
is clear that for any $\tau\leq\sigma$ the triple
$(Y,Z,K-K_{\tau})$ is also a solution of the local RBSDE on
$[\tau,\sigma]$. It is well known that in the space $Conv\cap
B(0,N)$ there exists a countable dense set $\{G_1,G_2,\dots\}$ of
convex polyhedrons such that $G_i\subset B(0,N)$, $i\in\N$. By
what has already been proved for each $i\in\N$ there exists a
solution $(Y_i,Z_i,K_i)$  of the local RBSDE in $G_i$ with
terminal value $\zeta_i=\Pi_{G_i}(\zeta)$. Set
$C_1^j=\{\rho(G_1,D)\leq 1/j\}$ and
\[
C^j_i=\{\rho(G_i,D)\leq1/j,\,\rho(G_1,D)>1/j,\dots,
\rho(G_{i-1}^j,D)>1/j\},\quad i=2,3,\dots
\]
Furthermore, for  $j\in\N$ set
\[
\zeta^j=\sum_{i=1}^\infty\Pi_{G_i}(\zeta){\bf 1}_{C_i^j},
\quad  D^j=\sum_{i=1}^\infty{G_i}{\bf 1}_{C_i^j}.
\]
Since  $C_i^j\in\FF_\tau$ for $i\in\N$,
$(Y^j,Z^j,K^j)=\sum_{i=1}^\infty (Y_i,Z_i,K_i){\bf 1}_{C^j_i}$ is
a solution of the local RBSDE in $D^j$ and  terminal value
$\zeta^j$.  Set
\[
A^j=\left\{
\begin{array}{ll}
A,&\mbox{\rm if }\dist(A,\partial D)>1/j, \\
a_i\in\mbox{\rm Int} G_i, &\mbox{\rm if }
\dist(A,\partial D)\leq 1/j\,\,\mbox{\rm and}\,\,D^j=G_i,\,i\in\N,
\end{array}
\right.
\]
and observe that $|\zeta^j|\leq N$ and $|A^j|\leq N$, $j\in\N$.
Therefore, by Proposition \ref{cor3.4}, for any $j\in\N$,
\[
E\big(\sup_{\tau\leq t\leq\sigma}|Y^j_t|^2
+\int_\tau^{\sigma}\|Z^j_s\|^2\,ds\,\FF_{\tau}\big) \leq C
\Big(N^2+\int_\tau^\sigma|f(s,0,0)|^2\,ds\,|\,\FF_{\tau}\Big)
\]
and
\[
E\big(|K^j|_\tau^\sigma\,|\,\FF_\tau\big) \le C(\dist(A^j,\partial
D^j))^{-1} \Big(N^2
+E\int_\tau^\sigma|f(s,0,0)|^2\,ds\,|\,\FF_{\tau})\Big).
\]
Since $P(\dist(A,\partial D)>1/j)\uparrow1$ and
$\dist(A^j,\partial D^j)>\dist(A,\partial D)-1/j$ if
$\dist(A,\partial D)>1/j$, it follows that
\begin{equation}
\label{eq3.5}
\{E\big(|K^j|_\tau^\sigma\,|\,\FF_\tau\big);\,{j\in\N}\}
\quad\mbox{\rm is bounded in probability}.
\end{equation}
By Proposition \ref{cor3.5}, for any $j,k\in\N$ we have
\begin{align}
\label{eq3.06}
&E\big(\sup_{\tau\leq t\leq\sigma}|Y^j_t-Y^{j+k}_t|^2
+\int_\tau^\sigma\|Z^j_s-Z^{j+k}_s\|^2\,ds\,|\,\FF_\tau)\nonumber\\
&\qquad\le C\big(E(|\zeta^j-\zeta^{j+k}|^2\,|\,\FF_{\tau})
+\rho(D^j,D^{j+k})E(|K^j|_\tau^\sigma
+|K^{j+k}|_\tau^\sigma\,|\,\FF_{\tau})\big).
\end{align}
By the construction, $|\zeta^j-\zeta^{j+k}|\leq2/j$ and
$\rho(D^j,D^{j+k})\leq2/j$ for $k\in\N$ . Therefore from
(\ref{eq3.5}) and (\ref{eq3.06}) it follows that
$\{(Y^j,Z^j,K^j)\}$ is a Cauchy sequence on $[\tau,\sigma]$ in the
space ${\cal S}\times{\cal P}\times{\cal S}$. By using standard
methods we show that its limit $(Y,Z,K)$ is a solution of the
local RBSDE (\ref{eq3.2}).
\end{proof}
\medskip

Let us remark that in fact assumption (\ref{eq3.4}) in
Proposition \ref{prop3.6} is superfluous (see Remark
\ref{rem3.8}).

\begin{lemma}
\label{lem3.7} Let $\{G_t;\,t\in[0,T]\}$ be a family of bounded
closed subsets of $\R^m$  such that $t\mapsto G_t$ is c\`adl\`ag
with respect to the Hausdorff metric $\rho$ and $G_T=G_{T-}$.
Define its discretization $\{G^j_t;\,t\in[0,T]\}$ by putting
$G^j_t=G_{t^j_{i-1}}$, $t\in[t^j_{i-1},t^j_{i})$,
$G^j_T=G^j_{T-}$\,, where $t^j_{0}=0$,
$t^j_{i}=(t^j_{i-1}+1/j)\wedge\inf
\{t>t^j_{i-1};\,\rho(G_{t-},G_t)>1/j\}\wedge T$, $i,j\in\N$. Then
$\sup_{t\leq T}\rho(G_t,G^j_t)\rightarrow0$ as
$j\rightarrow\infty$.
\end{lemma}
\begin{proof}
Suppose the assertion of the lemma is false. Then there exists
$t\in[0,T]$ and a sequence $\{t_j\}$ such that $t_j\to t$ and
\begin{equation}
\label{eq3.6}
\rho(G_{t_j},G^j_{t_j})\not\to0.
\end{equation}
Observe that for each $j\in\N$, $G^j_{t_j}=G_{s_j}$, where
$s_j=\max\{t^j_{i};\,t^j_{i}\leq t_j\}$. Since $t_j-1/j\leq
s_j\leq t_j$,  $s_j\to t$. If $\rho(G_{t-},G_t)=0$ then
\[
0\leq \rho(G_{t_j},G^j_{t_j})\leq \rho(G_{t_j},G_t)
+\rho(G_t,G_{s_j})\rightarrow0,
\]
which contradicts (\ref{eq3.6}). If $\rho(G_{t-},G_t)>0$ then
$t\in\{t^j_{i};\,i\in\N\cup\{0\}\}$ for sufficiently large $j$
(such that $\rho(G_{t-},G_t)>1/j$). Set $J^+=\{j;\,t_j\geq t\}$,
$J^-=\{j;\,t_j<t\}$ and assume that both sets are infinite. If
$t\leq t_j$ then for sufficiently large $j$,
$t=\max\{t^j_{i};\,t^j_{i}\leq t\}\leq s_j$. Consequently,
$\lim_{j\in J^+}\rho(G_{t_j},G_t)=0$ and $\lim_{j\in J^+}
\rho(G_{s_j},G_t)=0$. Since $s_j\leq t_j$, $\lim_{j\in J^-}
\rho(G_{t_j},G_{t-})=0$ and $\lim_{j\in
J^-}\rho(G_{s_j},G_{t-})=0$. Hence
\begin{align*}
0\leq\rho(G_{t_j},G^j_{t_j})&\leq (\rho(G_{t_j},G_t)
+\rho(G_t,G_{s_j}))\mbox{\bf 1}_{\{j\in J^+\}}\\
&\quad+(\rho(G_{t_j},G_{t-})+\rho(G_{t-},G_{s_j})\mbox{\bf
1}_{\{j\in J^-\}}\rightarrow0,
\end{align*}
which also contradicts (\ref{eq3.6}).
\end{proof}
\medskip

We are now ready to prove our main theorem of this section.

\begin{theorem}
\label{thm3.1} Assume \mbox{\rm(H1)--(H4)}. Then there exists a
unique solution $(Y,Z,K)$  of the  RBSDE \mbox{\rm(\ref{eq1.1})}
such that $Y,K\in{\cal S}^2$ and  $Z\in{\cal P}^2$.
\end{theorem}
\begin{proof}
{\em Step 1.} We begin by proving the theorem under the additional
assumption that there exists $N\in\N$ such that
\begin{equation}
\label{eq3.7} D_t\subset B(0,N),\quad t\in[0,T].
\end{equation}
For $j\in\N$ set $\sigma^j_{0}=0$ and
\[
\sigma^j_{i}=(\sigma^j_{i-1}+1/j)\wedge
\inf\{t>\sigma^j_{i-1};\rho(D_{t-},D_t)>1/j\}\wedge T, \quad
i\in\N.
\]
Since $t\to D_t$ is c\`adl\`ag, for every $j\in\N$ there is $k_j$
such that $P(\sigma^j_{k_j}<T)\leq1/j$. Set
\[
D^j_t=\left\{
\begin{array}{ll}
D_{\sigma^j_{i-1}},& t\in[\sigma^j_{i-1},
\sigma^j_{i}[,\,i=1,\dots,k_j-1,\\
D_{\sigma^j_{k_j}},& t\in[\sigma^j_{k_j},T].
\end{array}
\right.
\]
Then
\begin{equation}
\label{eq3.8}
\sup_{t\leq T}\rho(D^j_t,D_t)\arrowp0
\end{equation}
as $j\rightarrow\infty$. Indeed, by Lemma \ref{lem3.7},
$\sup_{t\leq \sigma^j_{k_j}}\rho(D^j_t,D_t)\to0$ $P$-a.s.
Therefore for every $\varepsilon>0$,
\[
P(\sup_{t\leq T}\rho(D^j_t,D_t)>\varepsilon) \leq P(\sup_{t\leq
\sigma^j_{k_j}}\rho(D^j_t,D_t)>\varepsilon)+1/j \rightarrow0.
\]
By the above and (H4) one can find a sufficiently slowly
decreasing sequence $\delta_j\downarrow0$ such that the sequence
$\{\gamma_j\}$ defined as
\[
\gamma_j=\inf\{t;\,\dist(A_t,\partial D^j_t)<\delta_j\}\wedge
T,\quad j\in\N
\]
has the property that $P(\gamma_j<T)\rightarrow0$. By Propositions
\ref{prop3.2} and \ref{prop3.6} for each $j\in\N$ there exists a
solution $(Y^j,Z^j,K^j)$ of RBSDE in the stopped time-dependent
region
$\DD^{j,\gamma_j-}=\{D^{j,\gamma_j-}_t=D^j_{t\wedge(\gamma_j-)};\,t\in[0,T]\}$
with terminal value $\xi^j=\Pi_{D^j_{\gamma_{j}-}}(\xi)$. Set
$A^j_t=A^{\gamma_j-}_t$, $ t\in[0,T]$, and observe that
$\inf_{t\leq T}\dist(A^j_t,\partial D^j_t)\geq\delta_j>0$. Since
for any predictable locally bounded process $H$,
\[
(\int_0^{\cdot}\langle H_s,\,dA^j_s\rangle)^*_T=(\int_0^{\cdot}
\langle H_s,\,dA_s\rangle)^*_{\gamma_j-}\,,
\]
it follows from Remark \ref{rem2.1} that there is $c>0$ such that
$\|A^j\|_{{\cal H}^2}\leq c\|A\|_{{\cal H}^2}$, $j\in\N$. Hence,
by Proposition \ref{prop2.4}, there exists $C>0$ such that for
every $j\in\N$,
\begin{align*}
&E\big(\sup_{t\leq T}|Y^j_t|^2+\int_0^T\|Z^j_s\|^2\,ds
+\sum_{s\leq T}|\Delta K^j_s|^2 +\int_0^T\dist(A^j_{s-},\partial
D^j_{s-})\,d|K^j|_s\big)\\
&\qquad\le C \Big(N^2+E\int_0^T|f(s,0,0)|^2\,ds +\|A\|^2_{{\cal
H}^2}\Big).
\end{align*}
For every $\varepsilon>0$ there is $M>0$,
a stopping time $\sigma_j\leq T$ and $j_0\in\N$  such that for
every $j\geq j_0$,
\begin{equation}
\label{eq3.9} P(\sigma_j< T)\leq \varepsilon,\quad
|K^j|_{\sigma_j-}\leq M.
\end{equation}
Indeed, by (H4) there is $\delta>0$ such that $P(\inf_{t\leq T}
\mbox{\rm dist}(A_t,\partial D_t)\leq\delta)\leq {\varepsilon}/4$.
On the other hand, by (\ref{eq3.8}), there is $j_0$ such that for
$j\geq j_0$, $P(\sup_{t\leq T}
\rho(D^j_t,D_t)>\delta)\leq\varepsilon/4$. Therefore for every
$j\geq j_0$,
\[
P(\inf_{t\leq T}\mbox{\rm dist}(A^j_t,\partial D^j_t)\leq\delta)
\leq P( \inf_{t\leq T}\mbox{\rm dist}(A_t,\partial
D_t)\leq2\delta) + P(\sup_{t\leq T}
\rho(D^j_t,D_t)>\delta)\leq\frac{\varepsilon}2\,.
\]
Set $c=C\big(E(N^2+\int_0^T|f(s,0,0)|^2\,ds)+\|A\|^2_{{\cal
H}^2}\big)$,  $M=(2c)/(\varepsilon\delta)$ and
$\sigma_j=\inf\{t;|K^j|_t>M\}\wedge T$. By Proposition
\ref{prop2.4} and Tchebyshev's inequality,
\begin{align*}
P(\sigma_j< T)&\leq P(|K^j|_T>M) \leq P(|K^j|_T>M,\,\inf_{t\leq T}
\mbox{\rm dist}(A^j_t,\partial D^j_t)>\delta)+\frac\varepsilon2\\
&\leq P(\inf_{t\leq T}
\mbox{\rm dist}(A^j_t,\partial D^j_t)|K^j|_T>M\delta)+\frac\varepsilon2\\
&\leq P(\int_0^T\mbox{\rm dist}
(A^j_{s-},\partial D^j_{s-})\,d|K^j|_s>M\delta)+\frac\varepsilon2\\
&\leq\frac{c}{M\delta}+\frac\varepsilon2\le\varepsilon.
\end{align*}
If we set $\sigma=\sigma_j\wedge\sigma_{j+k}\wedge\gamma_j$ then
by Proposition \ref{prop2.5},
\begin{align*}
&E\big((\sup_{ t<\sigma}|Y^j_t-Y^{j+k}_t|^2
+\int_0^{\sigma}\|Z^j_s-Z^{j+k}_s\|^2\,ds\big)\\
&\qquad\leq C\big( E(| Y^j_{\sigma-}-Y^{j+k}_{\sigma-}|^2
+\int_0^{\sigma-}
\rho(D^j_{s-},D^{j+k}_{s-})\,d(|K^j|_s+|K^{j+k}|_s)\big)\\
&\qquad\leq C\big(E(|\xi^j-\xi^{j+k}|^2+2\varepsilon N^2
+2M\min(\sup_{s\leq T}\rho(D^j_{s-},D^{j+k}_{s-}),N)\big).
\end{align*}
Since $\lim_{j\to\infty}\sup_kE|\xi^j-\xi^{j+k}|^2=0$ and by
(\ref{eq3.8}),
\[
\lim_{j\to\infty}\sup_kE\min(\sup_{s\leq T}
\rho(D^j_{s-},D^{j+k}_{s-}),N)=0,
\]
it follows that $\{(Y^j,Z^j,K^j)\}$ is a Cauchy sequence
in ${\cal S}\times{\cal P}\times{\cal S}$. Its limit $(X,Z,K)$
is a solution of RBSDE (\ref{eq1.1}).
\medskip\\
{\em Step 2.} We will show  how to dispense with assumption
(\ref{eq3.7}). Set $\gamma_j=\inf\{t\ge0:\sup_{s\leq t}
|A_s|>N_j\}\wedge T$, $j\in\N$, where $N_j\uparrow \infty$ and
\[
D^j_t=D^{\gamma_j-}_t\cap B(A^{\gamma_j-}_t,N_j),\quad t\in[0,T].
\]
%\begin{equation*}
%D^j_t=\left\{
%\begin{array}{ll}
%D_t\cap B(A_t,m_j),&\mbox{\rm if } t<\tau_j \\
%B(0,1), &\mbox{\rm otherwise }
%\end{array}
%\right.
%\end{equation*}
Clearly $D^j_t\subset B(0,2N_j)$ and $P(\gamma_j<T)\leq
P(\sup_{t\leq T}|A_t|>N_j)\downarrow0$. By Step 1 for each
$j\in\N$ there exists a solution $(Y^j,Z^j,K^j)$ of RBSDE in
$\{D^j_t;\,t\in[0,T]\}$  with terminal value
$\xi^j=\Pi_{D^j_T}(\xi)$. Set $A^j_t=A^{\gamma_j-}_t$,
$t\in[0,T]$.
%\begin{equation*}
%A^j_t=\left\{\begin{array}{ll}
%A_t,&\mbox{\rm if } t<\tau_j \\
%0, &\mbox{\rm otherwise. }
%\end{array}
%\right.
%\end{equation*}
Since by Remark \ref{rem2.1} there is $c>0$ such that
$\|A^j\|_{{\cal H}^2}\leq c\|A\|_{{\cal H}^2}$ for $j\in\N$,  using
Proposition  \ref{prop2.4} we obtain
\begin{align*}
& E\big(\sup_{t\leq T}|Y^j_t|^2+\int_0^T\|Z^j_s\|^2\,ds
+\sum_{s\leq T}|\Delta K^j_s|^2 +\int_0^T\dist(A^j_{s-},\partial
D^j_{s-})\,d|K^j|_s\big)\\
&\qquad\le C \Big(E\big(\xi^2+\int_0^T|f(s,0,0)|^2\,ds\big)
+\|A\|_{{\cal H}^2}\Big).
\end{align*}
Set  $\tau_{j,k}=\inf\{t;\sup_{s\leq t}|Y^{j+k}_s|>2N_j\}\wedge T$
for $j,k\in\N$ and observe that by Tschebyshev's inequality,
\[
P(\tau_{j,k}<T)\leq P(\sup_{t\leq T}|Y^{j+k}_t|>2N_j) \leq
(2N_j)^{-2}C \Big(E\big(\xi^2+\int_0^T|f(s,0,0)|^2\,ds\big)
+|\|A\|_{{\cal H}^2}\Big),
\]
which implies that $\lim_{j\to\infty}\sup_kP(\tau_{j,k}<T)=0$. Let
$\sigma=\tau_{j,k}$. Since $Y^j_t\in D^{j+k}_t$ for $t\in[0,T]$
and $Y^{j+k}\in D^j_t$ for $t<\sigma$, from Proposition
\ref{prop2.5}  it follows that for $p<2$,
\begin{align*}
&E\big( \sup_{ t<\sigma}|Y^j_t-Y^{j+k}_t|^p\big)
\leq CE(| Y^j_{\sigma-}-Y^{j+k}_{\sigma-}|^p)\\
&\qquad\leq E|\xi^j-\xi^{j+k}|^p
+E| Y^j_{\sigma-}-Y^{j+k}_{\sigma-}|^p{\bf 1}_{\{\sigma<T\}}\\
&\qquad\leq E|\xi^j-\xi^{j+k}|^p
+(E|Y^j_{\sigma-}-Y^{j+k}_{\sigma-}|^2)^{p/2}
(P(\sigma<T))^{(2-p)/2}.
\end{align*}
Hence $\lim_{j\to\infty}\sup_kE\sup_{ t<\sigma}|
Y^j_{t}-Y^{j+k}_{t}|^p=0$ for $p<2$ from which we deduce that
$\{(Y^j,Z^j,K^j)\}_{j\in\N}$ is a Cauchy sequence in ${\cal
S}\times{\cal P}\times{\cal S}$. Using standard arguments one
can show that its limit $(X,Z,K)$ is a solution of (\ref{eq1.1}).
\end{proof}

\begin{remark}
Arguing as in Step 2 of the above proof one can dispense with
assumption (\ref{eq3.4}) in Proposition \ref{prop3.6}. Therefore
under (H1${}^*$)--(H4${}^*$) there exists  a unique solution
$(Y,Z,K-K_\tau)$ of \mbox{\rm(\ref{eq3.2})} such that
(\ref{eq3.05}) is satisfied.
\end{remark}

\begin{remark}
\label{rem3.8}
The assumption that $D_T=D_{T-}$ in Theorem
\ref{thm3.1} is superfluous, because if $D_T\neq D_{T-}$ then from
Theorem \ref{thm3.1} it follows that there exists a unique
solution of the RBSDE
\[
Y_t=\Pi_{D_{T-}}(\xi)+\int^T_tf(s,Y_s,Z_s)\,ds
-\int^T_tZ_s\,dW_s+K_T-K_t,\quad t\in[0,T]
\]
in $\{D^{T-}_t;\, t\in[0,T]\}$. Set
\[
Y'_t=\left\{\begin{array}{ll}
Y_t,&\mbox{\rm if } t<T, \\
\xi, &\mbox{\rm if } t=T,
\end{array}
\right.\quad Z=Z',\quad K'_t=\left\{\begin{array}{ll}
K_t,&\mbox{\rm if } t<T, \\
K_T+\Pi_{D_{T-}}(\xi)-\xi, &\mbox{\rm if } t=T.
\end{array}
\right.
\]
Then $(Y',Z',K')$ is a unique solution of RBSDEs in
$\{D_t;\,t\in[0,T]\}$.
\end{remark}

\begin{remark}
\label{rem3.10} In Step 1 of the proof of Theorem \ref{thm3.1} and
in Lemma \ref{lem3.7} one can use stopping times $\sigma^j_{i}$
defined as follows: $\sigma^j_{0}=0$ and
\[
\sigma^j_{i}=(\sigma^j_{i-1}+a^j_{i})
\wedge\inf\{t>\sigma^j_{i-1};\rho(D_{t-},D_t)>1/j\} \wedge T,\quad
i,j\in\N,
\]
where $a^j_{i}$ is an arbitrary constant such that $1/j\leq
a^j_{i}\leq 2/j$. This follows from the fact that if we use the
modified stopping times $\sigma^j_{i}$ to  define the process
$D^j$ then (\ref{eq3.8}) still holds true. We will use this simple
observation in the next section.
\end{remark}

\nsubsection{Approximation of solutions of RBSDEs by the modified
penalization method} \label{sec4}

We start with a priori estimates for solutions of the penalized
BSDEs  and their local versions.

\begin{proposition}
\label{prop4.1} Assume \mbox{\rm(H1)--(H4)}. If $(Y^n,Z^n,K^n)$ is
a solution of \mbox{\rm(\ref{eq1.2})} such that $Y^n\in{\cal S}^2$
then there exists $C>0$ depending only on $\mu,\lambda, T$ such
that
\begin{align*}
&E\Big((Y^{n,*}_T)^2+\int_0^T\|Z^n_s\|^2\,ds +\sum_{0<s\leq T}
|\Delta K^n_s|^2+(K^{n,*}_T)^2
+\int_0^T\dist(A_{s-}\,,\partial D_{s-})\,d|K^n|_s\Big) \\
&\qquad\leq C\Big(E\big(|\xi|^2+\int_0^T|f(s,0,0)|^2\,ds\big)
+\|A\|^2_{{\cal H}^2}\Big).
\end{align*}
\end{proposition}
\begin{proof}
The proof is similar to that of Proposition \ref{prop2.4}. To get
the desired estimate it suffices  to repeat step by step arguments
from the proof of Proposition \ref{prop2.4}, the only difference
being in the fact that to obtain an analogue of (\ref{eq2.02}) we
have to prove that
\begin{equation}
\label{eq4.1} \int_0^t\langle
Y^n_{s-}-A_{s-},dK^{n}_s\rangle \leq-\int_0^t\dist(A_{s-},\partial
D_{s-})\,d|K^{n}|_s, \quad t\in[0,T].
\end{equation}
To prove (\ref{eq4.1}) let us define $K^{n,d}$, $K^{n,c}$ by
(\ref{eq1.3}).
%\[
%K^{n,d}_t=-\sum_{\sigma_{n,i}\leq
%t}(Y^n_{\sigma_{n,i}}-\Pi_{D_{{\sigma_{n,i}}-}}
%(Y^n_{\sigma_{n,i}})),\quad  K^{n,c}_t=K^n_t-K^{n,d}_t,\quad
%t\in[0,T].
%\]
Observe that by Remark \ref{rem2.2}(c),
\begin{align*}
\int_0^t\langle Y^n_{s-}-A_{s-},dK^{n,c}_s\rangle
&=n\int_0^t\langle Y^n_{s-}-A_{s-},
\Pi_{D_s}(Y^n_s)-Y^n_s\rangle\,ds\\
&\leq-n\int_0^t \dist(A_{s-},\partial D_{s-})
|\Pi_{D_s}(Y^n_s)-Y^n_s|\,ds\\
&=-\int_0^t \dist(A_{s-},\partial D_{s-})\,d|K^{n,c}|_s\,.
\end{align*}
By Remark \ref{rem2.2}(b), for  $i=1,2,\dots,k_n$ we have
\begin{align*}
\langle Y^n_{\sigma_{n,i}-}-A_{\sigma_{n,i}-}, \Delta
K^n_{\sigma_{n,i}}\rangle &= \langle
Y^n_{\sigma_{n,i}-}-A_{\sigma_{n,i}-},
\Pi_{D_{\sigma_{n,i}-}}(Y^n_{\sigma_{n,i}})
-Y^n_{\sigma_{n,i}}\rangle\\
&\leq-\dist(A_{\sigma_{n,i}-},\partial D_{\sigma_{n,i}-})
|\Pi_{D_{\sigma_{n,i}-}}(Y^n_{\sigma_{n,i}})-Y^n_{\sigma_{n,i}}|\\
&=-\dist(A_{\sigma_{n,i}-},\partial D_{\sigma_{n,i}-})|\Delta
K^n_{\sigma_{n,i}}|.
\end{align*}
Putting together the above two estimates we get (\ref{eq4.1}).
\end{proof}
\medskip

Let $\xi'\in L^2$ and let $\DD'=\{D'_t,t\in[0,T]\}$ be a family
satisfying (H4)  with some semimartingale $A'$. In the next
proposition we consider RBSDE in $\DD'$ of the form
\begin{equation}
\label{eq4.2}
Y'^n_t=\xi'+\int^T_tf(s,Y'^n_s,Z'^n_s)\,ds
-\int^T_t Z'^n_s\,dW_s+K'^n_T-K'^n_t,\quad
t\in[0,T],
\end{equation}
where
\[
K'^n_t=-n\int_0^t (Y'^n_s-\Pi_{D'_s}(Y'^n_s))\,ds
-\sum_{\sigma'_{n,k}\leq
t}(Y'^n_{\sigma'_{n,k}}-\Pi_{D'_{{\sigma'_{n,k}}-}}
(Y'^n_{\sigma'_{n,k}})) ,\quad t\in[0,T]
\]
and $\sigma'_{n,0}=0$,
$\sigma'_{n,i}=\inf\{t>\sigma'_{n,i-1};\,\rho(D'_t\cap
B(0,n),D'_{t-}\cap B(0,n))>1/n\}\wedge T$, $i=1,\dots,k'_n$ for
some $k_n'\in\N$.

\begin{proposition}
\label{prop4.2} Let $(Y^n,Z^n,K^n)$, $(Y'^{n},Z'^{n},K'^{n})$ be
solutions of \mbox{\rm(\ref{eq1.2})} and \mbox{\rm(\ref{eq4.2})},
respectively, such that $Y^n,Y'^{n}\in{\cal S}^2$. Set $\bar
Y^n=Y^n-Y'^{n}$, $\bar Z^n=Z^n-Z'^{n}$, $\bar K^n=K^n-K'^{n}$. If
$f$ satisfies \mbox{\rm(H3)} then for every $p\in(1,2]$ there
exists $C>0$ depending only on $\mu,\lambda,T$ such that for every
stopping time $\sigma$ such that $0\leq\sigma\leq T$,
\begin{align*}
&E\Big(\sup_{t<\sigma}|\bar Y^n_t|^p +\int_0^\sigma|\bar
Y^n_s|^{p-2}{\bf 1}_{\{\bar Y^n_s\neq0\}}\|\bar Z^n_s\|^2\,ds
+I^n_{\sigma-}\Big)\\
&\quad\leq CE\Big(|\bar Y^n_{\sigma-}|^p
+\int_0^{\sigma-}|\bar Y^n_{s-}|^{p-2}
|\Pi_{D_{s-}}(\Pi_{D'_{s-}}(Y'^{n}_{s-}))
-\Pi_{D'_{s-}}(Y'^{n}_{s-})|\,d|K^{n,c}|_s\\
&\qquad+\int_0^{\sigma-}|\bar Y^n_{s-}|^{p-2}
|\Pi_{D'_{s-}}(\Pi_{D_{s-}}(Y^n_{s-}))
-\Pi_{D_{s-}}(Y^n_{s-})|\,d|K'^{n,c}|_s\\
&\qquad+\int_0^{\sigma-}|\bar Y^n_{s-}|^{p-2}
(|\Pi_{D_{s-}}(Y'^{n}_{s-})-Y'^{n}_{s-}|\,d|K^{n,d}|_s
+|\Pi_{D'_{s-}}(Y^n_{s-})-Y^n_{s-}|\,d|K'^{n,d}|_s)\Big),
\end{align*}
where $I^n_t=\sum_{s\leq t}(|\bar Y^n_s|^p-|\bar
Y^n_{s-}|^p-p|\bar Y^n_{s-}|^{p-1}\langle \sgn(\bar
Y^n_{s-},\Delta \bar Y^n_s\rangle)$, $t\ge0$.
\end{proposition}
\begin{proof}
The proof is similar to that of Proposition \ref{prop2.5}. We
first apply It\^o's formula to the function $x\to|x|^p$ and the
semimartingale $Y^n$ to get an analogue of (\ref{eq2.09}). Then we
estimate the terms of the right-hand side of the equality thus
obtained in much the same way as in the proof of Proposition
\ref{prop2.5}, except for an analogue of (\ref{eq2.8}). Now
\[
\int_t^{\sigma}| \bar Y^n_{s-}|^{p-1}\langle\sgn(\bar
Y^n_{s-},d\bar K^n_s\rangle=\int_t^{\sigma}| \bar
Y^n_{s-}|^{p-2}{\bf 1}_{\{Y^n_s\ne Y'^n_s\}}\langle \bar
Y^n_{s-},d\bar K^n_s\rangle
\]
and instead of (\ref{eq2.8}) we have to show that
\begin{align}
\label{eq4.3} \langle \bar Y^n_{s-},d\bar K^n_s\rangle
&\leq|\Pi_{D_{s-}}(\Pi_{D'_{s-}}(Y'^n_{s-}))
-\Pi_{D'_{s-}}(Y'^n_{s-})|d|K^{n,c}|_s\nonumber\\
&\quad+ |\Pi_{D'_{s-}}(\Pi_{D_{s-}}(Y^n_{s-}))
-\Pi_{D_{s-}}(Y^n_{s-})|d|K'^{n,c}|_s\nonumber\\
&\quad+|\Pi_{D_{s-}}(Y'^{n}_{s-})-Y'^{n}_{s-}|\,d|K^{n,d}|_s
+|\Pi_{D'_{s-}}(Y^n_{s-})-Y^n_{s-}|\,d|K'^{n,d}|_s.
\end{align}
To see this, we first observe that
\begin{align*}
\langle\bar Y^n_{s-},d\bar K^{n,c}_s\rangle
&=\langle \bar Y^n_{s-}-\Pi_{D_{s-}}(Y^n_{s-})
+\Pi_{D'_{s-}}(Y'^n_{s-}),d\bar K^{n,c}_s\rangle\\
&\quad+
\langle\Pi_{D_{s-}}(Y^n_{s-})
-\Pi_{D'_{s-}}(Y'^n_{s-}),d\bar K^{n,c}_s\rangle\\
&\leq\langle \Pi_{D_{s-}}(Y^n_{s-})-\Pi_{D'_{s-}}(Y'^n_{s-}),d\bar
K^{n,c}_s\rangle,
\end{align*}
because
\begin{align*}
&\langle \bar Y^n_{s-}-\Pi_{D_{s-}}(Y^n_{s-})
+\Pi_{D'_{s-}}(Y'^n_{s-}),d\bar K^{n,c}_s\rangle\\
&\qquad=-n \langle \bar Y^n_{s-}-\Pi_{D_{s-}}(Y^n_{s-})
+\Pi_{D'_{s-}}(Y'^n_{s-}),\bar Y^n_{s}-\Pi_{D_{s}}(Y^n_{s})
+\Pi_{D'_{s}}(Y'^n_{s})\rangle\,ds\\
&\qquad=-n|\bar Y^n_{s}-\Pi_{D_{s}}(Y^n_{s})
+\Pi_{D'_{s}}(Y'^n_{s})|^2\,ds \leq0.
\end{align*}
By Remark \ref{rem2.2}(b),
\begin{align*}
&\langle \Pi_{D_{s-}}(Y^n_{s-})-\Pi_{D'_{s-}}(Y'^n_{s-}),
dK^{n,c}_s\rangle\\
&\qquad=-n\langle \Pi_{D_{s-}}(Y^n_{s-})
-\Pi_{D_{s-}}(\Pi_{D'_{s-}}(Y'^n_{s-})),Y^n_s-\Pi_{D_{s}}
(Y^n_{s})\rangle\,ds\\
&\qquad\quad+\langle \Pi_{D_{s-}}(\Pi_{D'_{s-}}
(Y'^n_{s-}))-\Pi_{D'_{s-}}(Y'^n_{s}),dK^{n,c}_s\rangle\\
&\qquad\leq \langle \Pi_{D_{s-}}(\Pi_{D'_{s-}}(Y'^n_{s-}))
-\Pi_{D'_{s-}}(Y'^n_{s}),dK^{n,c}_s\rangle.
\end{align*}
Using similar estimate for $\langle
\Pi_{D_{s-}}(Y^n_{s-})-\Pi_{D'_{s-}}
(Y'^n_{s-}),-dK'^{n,c}_s\rangle$ we obtain
\begin{align}
\label{eq4.03} \langle \bar Y^n_{s-},d\bar K^{n,c}_s\rangle
&\leq|\Pi_{D_{s-}}(\Pi_{D'_{s-}}(Y'^n_{s-}))
-\Pi_{D'_{s-}}(Y'^n_{s-})|d|K^{n,c}|_s\nonumber\\
&\quad+|\Pi_{D'_{s-}}(\Pi_{D_{s-}}(Y^n_{s-}))
-\Pi_{D_{s-}}(Y^n_{s-})|d|K'^{n,c}|_s.
\end{align}
On the other hand, by Remark \ref{rem2.2}(a), for
$i=1,2,\dots,k_n$ we have
\begin{align*}
&\langle Y^n_{\sigma_{n,i}-}-\Pi_{D_{\sigma_{n,i}-}}
(Y'^n_{\sigma_{n,i}-}),\Delta K^n_{\sigma_{n,i}}\rangle\\
&\qquad=\langle Y^n_{\sigma_{n,i}-}
-\Pi_{D_{\sigma_{n,i}-}}(Y'^n_{\sigma_{n,i}-}),
\Pi_{D_{\sigma_{n,i}-}}
(Y^n_{\sigma_{n,i}})-Y^n_{\sigma_{n,i}}\rangle\leq0
\end{align*}
and
\[
\langle
Y'^n_{\sigma'_{n,i}-}-\Pi_{D'_{\sigma'_{n,i}-}}
(Y^n_{\sigma'_{n,i}-}),\Delta
K'^n_{\sigma'_{n,i}}\rangle\leq0,
\]
which implies that
\begin{equation}
\label{eq4.04} \langle \bar Y^n_{s-},\Delta \bar
K^n_{s}\rangle\leq\langle \Pi_{D_{s-}}(Y'^n_{s-})-Y'^n_{s-},\Delta
K^n_{s}\rangle + \langle \Pi_{D'_{s-}}(Y^n_{s-})-Y^n_{s-},\Delta
K'^n_{s}\rangle.
\end{equation}
Combining (\ref{eq4.03}) with (\ref{eq4.04}) yields (\ref{eq4.3}).
We leave the details of the rest of the proof to the reader.
\end{proof}

\begin{corollary}
\label{cor4.3} Under the assumptions of Proposition \ref{prop4.2},
if moreover $\DD^{\sigma-}=\DD'^{\sigma-}$, then
\[
E\Big(( \sup_{t<\sigma}|\bar Y^n_t|^p +\int_0^\sigma|\bar
Y^n_s|^{p-2}{\bf 1}_{\{\bar Y^n_s\neq0\}}\|\bar Z^n_s\|^2\,ds
+I^n_{\sigma-}\Big) \leq CE|\bar Y^n_{\sigma-}|^p.
\]
\end{corollary}
\begin{proof} Follows immediately from Proposition \ref{prop4.2}.
\end{proof}
\medskip

Note that in Propositions \ref{prop4.1}, \ref{prop4.2}  and
Corollary \ref{cor4.3} we do not assume that  $\xi\in D_T$,
$\xi'\in D'_T$.

We now turn to the approximation of local RBSDEs. Let
$\tau,\sigma$ be stopping times such that $0\leq\tau\leq\sigma\leq
T$, Let $D,D'$ be $\FF_\tau$-measurable random convex sets with
nonempty interiors and let $\zeta,\zeta'\in L^2$ be
$\FF_\sigma$-measurable random variables (we do not assume neither
that $\zeta\in D$ $P$-a.s. nor that $\zeta'\in D'$ $P$-a.s.). We
consider approximations  of the form
\begin{equation}
\label{eq4.4}
Y^n_t=\zeta+\int^{\sigma}_tf(s,Y^n_s,Z^n_s)\,ds
-\int^{\sigma}_tZ^n_s\,dW_s +K^n_{\sigma}-K^n_t,\quad
t\in[\tau,\sigma]
\end{equation}
and
\begin{equation}
\label{eq4.5} Y'^n_t=\zeta'+\int^{\sigma}_tf(s,Y'^n_s,Z'^n_s)\,ds
-\int^{\sigma}_tZ'^n_s\,dW_s +K'^{n}_{\sigma}-K'^{n}_t,\quad
t\in[\tau,\sigma],
\end{equation}
where
\begin{equation}
\label{eq4.8}
K^n_t=-n\int_{\tau}^t(Y^n_s-\Pi_{D}(Y^n_s))\,ds,\quad
K'^{n}_t=-n\int^t_{\tau}(Y'^n_s-\Pi_{D'}(Y'^n_s))\,ds, \quad
t\in[\tau,\sigma].
\end{equation}

\begin{corollary}
\label{cor4.4} Assume \mbox{\rm(H1${}^*$)--(H4${}^*$)}. Let
$(Y^n,Z^n,K^n-K^n_\tau)$ be a solution of \mbox{\rm(\ref{eq4.4})}
such that $\sup_{\tau\leq t\leq\sigma}|Y^n_t|\in L^2$. Then there
exists $C>0$ depending only on $\mu,\lambda,T$ such that for any
$n\in\N$
\[
E\Big(\sup_{\tau\leq t\leq\sigma}|Y^n_t|^2
+\int_\tau^\sigma\|Z^n_s\|^2\,ds\,|\,\FF_{\tau}\Big) \leq C
E\Big(|\zeta|^2+|A|^2
+\int_\tau^\sigma|f(s,0,0)|^2\,ds\,|\,\FF_{\tau}\Big)
\]
and
\[
E(|K^n|_\tau^\sigma\,|\,\FF_\tau) \le C(\dist(A,\partial D))^{-1}
E\Big(|\zeta|^2+|A|^2
+\int_\tau^\sigma|f(s,0,0)|^2\,ds\,|\,\FF_{\tau}\Big).
\]
\end{corollary}
\begin{proof}
Follows from the proof of Proposition \ref{prop4.1}.
\end{proof}

\begin{corollary}
\label{cor4.5} let $(Y^n,Z^n,K^n)$, $(Y'^n,Z'^n,K'^n)$ be
solutions of \mbox{\rm(\ref{eq4.4})} and \mbox{\rm(\ref{eq4.5})},
respectively, such that $\sup_{\tau\leq t\leq\sigma}
|Y^n_t|,\sup_{\tau\leq t\leq\sigma}|Y'^n_t|\in L^2$. If $f$
satisfies \mbox{\rm(H2)} then there exists $C>0$ depending only on
$\mu,\lambda,T$ such that for any $n\in\N$,
% \begin{eqnarray*}&&
%E\big(|Y-Y'^*_T)^2+\int_0^T||Z_s-Z'_s||^2\,ds
%+\sum_{s\leq T}|\Delta K_s-\Delta K'_s|^2\big)\\
%&&\qquad\qquad\qquad\leq CE\Big(|\xi-\xi'|^2
%+\int_0^T\rho(D_{s-},D'_{s-})\,d(|K|+|K'|)_s\Big).
%\end{eqnarray*}
\begin{align*}
&E\big(\sup_{\tau\leq t\leq\sigma}|Y^n_t-Y'^n_t|^2
+\int_\tau^\sigma\|Z^n_s-Z'^n_s\|^2\,ds\,|\,\FF_{\tau}\big)\\
&\qquad\le C\Big(E(|\zeta-\zeta'|^2\,|\,\FF_{\tau})
+E(\sup_{\tau\leq t\leq\sigma}
|\Pi_D(\Pi_{D'}(Y'^n_{t-}))-\Pi_{D'}(Y'^n_{t-})|\,
|K^n|_\tau^\sigma\,|\,\FF_{\tau})\\
&\qquad\quad+E(\sup_{\tau\leq t\leq\sigma}
|\Pi_{D'}(\Pi_D(Y^n_{t-}))-\Pi_D(Y^n_{t-})|\,|K'^n|_\tau^\sigma\,
|\,\FF_{\tau})\Big).
\end{align*}
\end{corollary}
\begin{proof}
Follows from the proof of Proposition \ref{prop4.2}.
\end{proof}

\begin{proposition}
\label{prop4.6} Assume \mbox{\rm(H1${}^*$)--(H4${}^*$)} and
\mbox{\rm(\ref{eq3.4})}. Then
\[
\sup_{\tau\leq t\leq \sigma}|Y^n_t-Y_t|\arrowp0,
\quad\int_\tau^\sigma||Z^n_s-Z_s||^2\,ds\arrowp0,\quad
\sup_{\tau\leq t\leq \sigma}|K^n_t-K_t|\arrowp0,
\]
where  $(Y,Z,K)$ is a unique solution of the local RBSDE
\mbox{\rm(\ref{eq3.2})}.
\end{proposition}
\begin{proof}
First  set $D=G$ for some fixed convex set with nonempty interior.
Consider approximations of the form
\[
Y^n_t=\zeta+\int_t^Tg(s,Y^n_s,Z^n_s)\,ds
-\int_t^TZ^n_s\,dW_s+K^n_T-K^n_t,\quad t\in[0,T]
\]
where $g(s,\cdot,\cdot)=f(s,\cdot,\cdot){\bf 1}_{[0,\sigma[}(s)$.
By \cite[Theorem 5.9]{PP}, $(Y^n,Z^n,K^n)\rightarrow(Y,Z,K)$ in
${\SSS}^2\times{\PP}^2\times{\SSS}^2$, where $(Y,Z,K)$ is a
solution of  RBSDEs of the form (\ref{eq3d1}) in $G$. Since
$Y^n_t=Y_t=\zeta$, $Z^n_t=Z_t=0$  and $K^n_T=K^n_t$ for $t\geq
\sigma$, it is clear that  for any $\tau\leq\sigma$,
$(Y^n,Z^n,K^n)$  converges in ${\cal S}^2\times{\cal P}^2
\times{\cal S}^2$ to the solution of our local RBSDE on
$[\tau,\sigma]$.

Now let us define $D^j,\zeta^j,A^j$, $j\in\N$ as in the proof of
Proposition \ref{prop3.6} (observe that $|\zeta^j|\leq N$ and
$|A^j|\leq N$, $j\in\N$) and by $(Y^{j,n},Z^{j,n},K^{j,n})$ denote
a solution  of the local BSDE
\[
Y^{j,n}_t=\zeta^j+\int^{\sigma}_tf(s,Y^{j,n}_s,Z^{j,n}_s)\,ds
-\int^{\sigma}_tZ^{j,n}_s\,dW_s
-n\int_t^{\sigma}(Y^{j,n}_s-\Pi_{D^j}(Y^{j,n}_s))\,ds,\quad
t\in[\tau,\sigma].
\]
Using the first part of the proof and arguments from the proof of
Proposition \ref{prop3.6} one can show that for every
$j\in\N$,
\begin{equation}\label{eq4.6}
(Y^{j,n},Z^{j,n},K^{j,n})\rightarrow
(Y^{(j)},Z^{(j)},K^{(j)})\quad\mbox{\rm in }{\cal S}\times{\cal
P}\times{\cal S},
\end{equation}
where $(Y^{(j)},Z^{(j)},K^{(j)})$ is a solution of the local RBSDE
in $D^j$ with terminal value $\zeta^j$. Since
$|\zeta^j-\zeta|\leq2/j$ and $\rho(D^j,D)\leq2/j$, from Corollary
\ref{cor4.5} it follows that
\begin{align}
\label{eq4.09}
&E\big(\sup_{\tau\leq t\leq\sigma}|Y^n_t-Y^{j,n}_t|^2
+\int_\tau^\sigma\|Z^n_s-Z^{j,n}_s\|^2\,ds\,|\,\FF_\tau)\nonumber\\
&\qquad\le C\big(E(|\zeta-\zeta^{j}|^2\,|\,\FF_{\tau})
+\rho(D,D^{j})E(|K^n|_\tau^\sigma
+|K^{j,n}|_\tau^\sigma\,|\,\FF_{\tau})\big)\nonumber\\
&\qquad\le C\big(\frac4{j^2} +\frac2{j}E(|K^n|_\tau^\sigma
+|K^{j,n}|_\tau^\sigma\,|\,\FF_{\tau})\big).
\end{align}
By Corollary \ref{eq4.3},
\[
E(|K^n|_\tau^\sigma\,|\,\FF_\tau) \le C(\dist(A,\partial D))^{-1}
E\Big(N^2+\int_\tau^\sigma|f(s,0,0)|^2\,ds\,|\,\FF_{\tau}\Big).
\]
Using once again  Corollary \ref{cor4.3} and the fact that on the
set $\{\dist(A,\partial D)>1/j\}$ we have $\dist(A^j,\partial
D^j)>\dist(A,\partial D)-1/j$, we conclude that
\[
E(|K^{j,n}|_\tau^\sigma\,|\,\FF_{\tau}) \le  C(\dist(A,\partial
D)-\frac1{j})^{-1}
E\Big(N^2+\int_\tau^\sigma|f(s,0,0)|^2\,ds\,|\,\FF_{\tau}\Big)
\]
on $\{\dist(A,\partial D)>1/j\}$.  Since $P(\dist(A,\partial
D)>1/j)\uparrow1$, it follows from (\ref{eq4.09}) that for every
$\varepsilon>0$,
\begin{equation}
\label{eq4.7} \lim_{j\to\infty}\limsup_{n\to\infty}
P\big(E\big(\sup_{\tau\leq t\leq\sigma}|Y^n_t-Y^{j,n}_t|^2
+\int_\tau^\sigma\|Z^n_s-Z^{j,n}_s\|^2\,ds\,|\,\FF_\tau\big)
\geq\varepsilon\big)=0.
\end{equation}
Combining (\ref{eq4.6}) with (\ref{eq4.7}) and the fact that
$\{(Y^{(j)},Z^{(j)},K^{(j)})\}$ converges in ${\cal S}\times{\cal
P}\times{\cal S}$ to the solution  $(Y,Z,K)$ of the local RBSDE in
$D$ we get the desired convergence results.
\end{proof}

\begin{remark}
\label{rem4.7} Proposition \ref{prop4.6} may be slightly
generalized to encompass different terminal values in the
approximation sequence. More precisely, let $\zeta_n\in L^2$ be a
sequence of $\FF_\sigma$-measurable random variables such that
$\zeta_n\to\zeta$ in $L^2$, where $\zeta\in D$ $P$-a.s. Let us
define $(\tilde Y^n,\tilde Z^n,\tilde K^n)$ by (\ref{eq4.5}),
(\ref{eq4.8}) but with $\zeta'$ replaced by $\zeta_n$. Then
Proposition \ref{prop4.6} holds true with $(Y^n,Z^n,K^n)$ replaced
by $(\tilde Y^n,\tilde Z^n,\tilde K^n)$. To see this it suffices
to observe in Corollary \ref{cor4.5} we do not assume that
$\zeta\in D$, $\zeta'\in D'$. Therefore for any $n\in\N$,
\[
E\Big(\sup_{\tau\leq t\leq\sigma}|Y^n_t-\tilde Y^n_t|^2
+\int_\tau^\sigma\|Z^n_s-\tilde Z^n_s\|^2\,ds\,|\,\FF_{\tau}\Big)
\le C E(|\zeta-\zeta_n|^2\,|\,\FF_{\tau}),
\]
which leads to the desired conclusion.
\end{remark}

\begin{proposition}
\label{prop4.8} Assume \mbox{\rm(H1${}^*$)--(H4${}^*$)}. Let
$0=\sigma_0\leq\sigma_1\leq\dots\leq\sigma_{k+1}=T$ be stopping
times and let $D^0,D^1,\dots,D^{k}$  be random closed convex
subsets in $\R^m$ such that $D^i$ is $\FF_{\sigma_i}$-measurable
and there is $m\in\N$ such that $D^i\subset B(0,N)$,
$i=1,\dots,k$. Let $(Y,Z,K)$ be a unique solution of RBSDE
\mbox{\rm(\ref{eq1.1})} in $\{D_t;\,t\in[0,T]\}$ such that
$D_t=D^{i-1}$, $t\in[\sigma_{i-1},\sigma_i[$, $i=1,\dots,{k+1}$,
$D_T=D_{T-}$\,, and let $(Y^n,Z^n,K^n)$  be a solution of
\mbox{\rm(\ref{eq1.2})}. Then
\[
(Y^n,Z^n,K^n)\arrowp (Y,Z,K)\quad in\,\,\,{\cal S}\times{\cal
P}\times{\cal S}.
\]
\end{proposition}
\begin{proof}
By Proposition \ref{prop4.6},
\[
\sup_{\sigma_k\leq t\leq T}|Y^n_t-Y_t|\arrowp0,
\quad\int_{\sigma_k}^T\|Z^n_s-Z_s\|^2\,ds\arrowp0
\]
and
\[
\sup_{\sigma_k\leq t\leq T}|n\int_{\sigma_k}^{t}
(Y^n_s-\Pi_{D^k}(Y^n_s))\,ds -(K_t-K_{\sigma_k})|\arrowp0.
\]
Since
\[Y^n_{\sigma_k-}=\left\{
\begin{array}{ll}
\Pi_{D^{k-1}}(Y^n_{\sigma_k}),&\mbox{\rm if }
\dist(D^k,D^{k-1})>1/n\,,n\ge N ,\\
Y^n_{\sigma_k}, &\mbox{\rm otherwise, }
\end{array}
\right.
\]
it is clear that $Y^n_{\sigma_k-}\arrowp
\Pi_{D^{k-1}}(Y_{\sigma_k})=Y_{\sigma_k-}$. Similarly, if
$Y^n_{\sigma_i-}\arrowp Y_{\sigma_i-}$ for $i=k,k-1,\dots,1$ then
by Proposition \ref{prop4.6},
\[
\sup_{\sigma_{i-1}\leq t\leq \sigma_i}|Y^n_t-Y_t|\arrowp0,
\quad\int_{\sigma_{i-1}}^{\sigma_i}\|Z^n_s-Z_s\|^2\,ds\arrowp0
\]
and
\[
\sup_{\sigma_{i-1}\leq t
\leq \sigma_i}|n\int_{\sigma_{i-1}}^{t}
(Y^n_s-\Pi_{D_k}(Y^n_s))\,ds-(K_t-K_{\sigma_{i-1}})|\arrowp0.
\]
Consequently,  $Y^n_{\sigma_{i-1}-}\arrowp
\Pi_{D^{i-2}}(Y_{\sigma_{i-1}})=Y_{\sigma_{i-1}-}$ for $i\geq 2$.
Using backward induction completes the proof.
\end{proof}

\begin{theorem}
\label{thm4.9} Assume \mbox{\rm(H1)--(H4)}. Let
$(\{Y^n,Z^n,K^n)\}$ be a sequence of solutions of
\mbox{\rm(\ref{eq1.2})}. Then
\[
(Y^n,Z^n,K^n)\arrowp (Y,Z,K)\quad\mbox{in }{\cal S}\times{\cal
P}\times{\cal S},
\]
where $(Y,Z,K)$ is a unique solution of \mbox{\rm(\ref{eq1.1})}.
\end{theorem}
\begin{proof}
{\em Step 1}. As in the proof of Theorem \ref{thm3.1} we first assume
additionally that (\ref{eq3.7}) is satisfied. For $j\in\N$ set
$\sigma_{j,0}=0$ and
\[
\sigma^j_{i}=(\sigma^j_{i-1}+a^j_{i})
\wedge\inf\{t>\sigma^j_{i-1};\rho(D_{t-},D_t)>1/j\}\wedge T,\quad
i\in\N,
\]
where $a^j_{i}\in[1/j,2/j]$ is a constant chosen via the following
procedure. Suppose that $\tau\equiv\sigma^j_{i-1}$ is such that
$\tau+1/j<T$. By Proposition \ref{prop4.1} there is $c>0$ such
that
\begin{align*}
&\int_0^{2/j}E(\dist(A_{\tau+s},\partial D_{\tau+s})
|Y^n_{\tau+s}-\Pi_{D_{\tau+s}}(Y^n_{\tau+s})|\,ds\\
&\qquad\leq E\int_0^{T}\dist(A_{s},\partial D_{s})
|Y^n_{s}-\Pi_{D_{s}}(Y^n_{s})|\,ds\le cn^{-1}
\end{align*}
for every $n\in\N$. Therefore we can find $s\in[1/j,2/j]$, which
we denote by $a^j_{i}$, such that $E(\dist(A_{\tau+s},\partial
D_{\tau+s})|Y^n_{\tau+s}-\Pi_{D_{\tau+s}}(Y^n_{\tau+s})|\rightarrow0$
as $n\rightarrow\infty$. Since $\dist(A_{\tau+s},\partial
D_{\tau+s})>0$,
$Y^n_{\tau+s}-\Pi_{D_{\tau+s}}(Y^n_{\tau+s})\arrowp0$ for
$s=a^j_{i}$.  It follows that the stopping times $\sigma^j_{i}$
have the property that
\begin{equation}
\label{eq4.9} Y^n_{\sigma_{j,i-1}+a_{i,j}}
-\Pi_{D_{\sigma^j_{i-1}+a^j_{j}}}(Y^n_{\sigma^j_{i-1}+a^j_{j}})\arrowp0,
\quad j,i\in\N.
\end{equation}
Now let us define
$\DD^j=\{D^j_t;\,t\in[0,T]\},\xi^j,A^j=\{A^j_t;\,t\in[0,T]\}$ as
in Step 1 of the proof Theorem \ref{thm3.1} (observe that
$|\xi^j|\leq N$ and $|A^j_t|\leq N$, $j\in\N$). Let
$(Y^{j,n},Z^{j,n},K^{j,n})$ denote the solution of the BSDE
\begin{equation}
\label{eq4.10}
Y^{j,n}_t=\xi^j+\int^{T}_tf(s,Y^{j,n}_s,Z^{j,n}_s)\,ds
-\int^{T}_tZ^{j,n}_s\,dW_s +K^{j,n}_T-K^{j,n}_t,\quad t\in[0,T],
\end{equation}
where
\[
K^{j,n}_t=-n\int_0^{t} (Y^{j,n}_s-\Pi_{D^j_{s}}(Y^{j,n}_s))\,ds
-\sum_{\sigma^j_{n,i}\leq
t}(Y^{j,n}_{\sigma^j_{n,i}}-\Pi_{D^j_{{\sigma^j_{n,i}}-}}
(Y^{j,n}_{\sigma^j_{n,i}})) ,\quad t\in[0,T]
\]
with $\sigma^j_{n,0}=0$,
$\sigma^j_{n,i}=\inf\{t>\sigma^j_{n,i-1};\,\rho(D^j_t,D^j_{t-})>1/n\}\wedge
T$, $i=1,\dots,k^j_n$ with $k^j_n$ chosen so that
$P(\sigma^j_{n,k^j_n}<T)\to 0$ as $n\to\infty$. From Proposition
\ref{prop4.8} we deduce that for each $j\in\N$,
\begin{equation}
\label{eq4.11}
(Y^{j,n},Z^{j,n},K^{j,n})\rightarrow (Y^{(j)},Z^{(j)},K^{(j)})
\quad\mbox{\rm in }{\cal S}\times{\cal P}\times{\cal S},
\end{equation}
where $(Y^{(j)},Z^{(j)},K^{(j)})$ is a solution of  RBSDE in
$\DD^j$ with terminal value $\xi^j$. Moreover,  by Proposition
\ref{prop4.1} there exists $C>0$ such that for $j,n\in\N$,
\begin{align*}
&E\Big(\sup_{t\leq T}|Y^{j,n}_t|^2
+\int_0^T\|Z^{j,n}_s\|^2\,ds+\sum_{s\leq T}|\Delta K^{j,n}_s|^2
+\int_0^T\dist(A^j_{s-}\,,\partial D^j_{s-})\,d|K^{j,n}|_s\Big)\\
&\qquad\le C\Big(N^2+E\int_0^T|f(s,0,0)|^2\,ds+\|A\|^2_{{\cal
H}^2}\Big).
\end{align*}
Consequently, by the same arguments as in the proof of Theorem
\ref{thm3.1} we deduce that for every $\varepsilon>0$ there exist
$M>0$,  stopping times $\sigma^n_j\leq T$ and $j_0\in\N$  such
that for every $j\geq j_0$ and $n\in\N$,
\begin{equation}
\label{eq4.12} P(\sigma^n_j< T)\leq \varepsilon,\quad
|K^{j,n}|_{\sigma^n_j-}\leq M.
\end{equation}
Similarly, by Proposition \ref{prop4.1} we show that for every
$\varepsilon>0$ there exist $M>0$ and stopping times $\tau^n\leq
T$ such that for every $n\in\N$,
\begin{equation}
\label{eq4.13} P(\tau^n< T)\leq
\varepsilon,\quad|K^{n}|_{\tau^n-}\leq M.
\end{equation}
Putting $p=2$ and $\sigma=\sigma^n_j\wedge\tau^n$ in Proposition
\ref{prop4.2} we obtain
\begin{align*}
&E\big(( \sup_{ t<\sigma}|Y^{j,n}_t-Y^{n}_t|^2
+\int_0^{\sigma}\|Z^{j,n}_s-Z^{n}_s\|^2\,ds\big)\\
&\qquad\leq C\Big(E(|Y^{j,n}_{\sigma-}-Y^{n}_{\sigma-}|^2
+\int_0^{\sigma-}\rho(D^j_{s-},D_{s-})\,d(|K^{j,n,c}|_s+|K^{n,c}|_s)\\
&\qquad\quad+\sum_{\sigma^j_{n,i}<\sigma}
|\Pi_{D^j_{\sigma^j_{n,i}-}}(Y^{n}_{\sigma^j_{n,i}-})-Y^{n}_{\sigma^j_{n,i}-}|
\cdot|\Delta K^{j,n}_{\sigma^j_{n,i}}|\\
&\qquad\quad+\sum_{\sigma_{n,i}<\sigma}
|\Pi_{D_{\sigma_{n,i}-}}(Y^{j,n}_{\sigma_{n,i}-})-Y^{j,n}_{\sigma_{n,i}-}|
\cdot|\Delta K^{n}_{\sigma_{n,i}}|\Big).
\end{align*}
If $\rho(D_{\sigma^j_{n,i}},D_{\sigma^j_{n,i}-})>0$ then
$Y^n_{\sigma^j_{n,i}-}\in D_{\sigma^j_{n,i}-}$ and
\[
|\Pi_{D^j_{\sigma^j_{n,i}-}}(Y^{n}_{\sigma^j_{n,i}-})
-Y^{n}_{\sigma^j_{n,i}-}|
\leq\rho(D^j_{\sigma^j_{n,i}-},D_{\sigma^j_{n,i}-})
\]
for $n\ge\max(j,N)$. Using the above estimate if
$\rho(D_{\sigma^j_{n,i}},D_{\sigma^j_{n,i}-})>0$ and (\ref{eq4.9})
if $\rho(D_{\sigma^j_{n,i}},D_{\sigma^j_{n,i}-})=0$ we obtain
\begin{align*}
&\limsup_{n\to\infty}E\big(\sum_{\sigma^j_{n,i}<\sigma}
|\Pi_{D^j_{\sigma^j_{n,i}-}}(Y^{n}_{\sigma^j_{n,i}-})
-Y^{n}_{\sigma^j_{n,i}-}|\cdot|\Delta K^{j,n}_{\sigma^j_{n,i}}|\big)\\
&\qquad\leq \limsup_{n\to\infty}E(\int_0^{\sigma-}
\rho(D^j_{s-}\,,D_{s-})\,d(|K^{j,n,d}|_s).
\end{align*}
Similarly,
\begin{align*}
&\limsup_{n\to\infty}E\big(\sum_{\sigma_{n,i}<\sigma}
|\Pi_{D_{\sigma_{n,i}}-}(Y^{j,n}_{\sigma_{n,i}-})
-Y^{j,n}_{\sigma_{n,i}-}|\cdot|\Delta K^{n}_{\sigma_{n,i}-}|\big)\\
&\qquad\leq \limsup_{n\to\infty}E(\int_0^{\sigma-}
\rho(D^j_{s-}\,,D_{s-})\,d(|K^{n,d}|_s).
\end{align*}
Hence
\begin{align*}
&\limsup_{n\to\infty}E\big(( \sup_{ t<\sigma}|Y^{j,n}_t-Y^{n}_t|^2
+\int_0^{\sigma}\|Z^{j,n}_s-Z^{n}_s\|^2\,ds\big)\\
&\qquad\leq CE\big(| \xi^j-\xi|^2+2\varepsilon N^2
+2M\min(\sup_{t\leq T}\rho(D^j_{t-},D_{t-}),N)\big).
\end{align*}
Consequently,
\begin{equation}
\label{eq4.14} \lim_{j\to\infty}
\limsup_{n\to\infty}E\big(\sup_{t<\sigma}|Y^{j,n}_t-Y^{n}_t|^2
+\int_0^{\sigma}\|Z^{j,n}_s-Z^{n}_s|^2\,ds\big)\leq 2C\varepsilon
N^2,
\end{equation}
because $\lim_{j\to\infty}\sup_kE|\xi^j-\xi|^2=0$ and
$E\sup_{t\leq T}\rho(D^j_{t-},D_{t-})\rightarrow0$ by
(\ref{eq3.7}), (\ref{eq3.8}) and Remark \ref{rem3.10}.
Furthermore, by the Step 1 of the proof of Theorem
\ref{thm3.1} and Remark \ref{rem3.10},
\begin{equation}
\label{eq4.18} (Y^{(j)},Z^{(j)},K^{(j)})\arrowp
(Y,Z,K)\quad\mbox{in }{\cal S}\times{\cal P}\times{\cal S},
\end{equation}
where $(Y,Z,K)$ is a unique solution of (\ref{eq1.1}).  Combining
(\ref{eq4.11}) with (\ref{eq4.14}) and (\ref{eq4.18}) we conclude
that  $(Y^n,Z^n,K^n)\arrowp (Y,Z,K)$ in ${\cal S}\times{\cal P}
\times{\cal S}$ under the additional assumption (\ref{eq3.7}).
\smallskip\\
{\em Step 2}. In the general case we will  use  arguments from
Step 2 of the proof of Theorem \ref{thm3.1}. Let $N_j$,
$\gamma_j$,
$\DD^j=\{D^j_t;\,t\in[0,T]\},\xi^j,A^j=\{A^j_t;\,t\in[0,T]\}$ be
defined as in that step. Note that $ D^j_t\subset B(0,2N_j)$. Let
$(Y^{j,n},Z^{j,n},K^{j,n})$ be a solution of (\ref{eq4.10}) with
$\sigma^j_{n,0}=0$,
$\sigma^j_{n,i}=\inf\{t>\sigma^j_{n,i-1};\,\rho(D^j_t\cap
B(0,n),D^j_{t-}\cap B(0,n))>1/n\}\wedge T$, $i=1,\dots,k^j_n$, and
$k^j_n$ is chosen so that $P(\sigma^j_{n,k^j_n}<T)\to 0$ as
$n\to\infty$. From the  first part of the proof we know that for
each $j\in\N$,
\begin{equation}
\label{eq4.19} (Y^{j,n},Z^{j,n},K^{j,n})\rightarrow (\tilde
Y^{(j)},\tilde Z^{(j)},\tilde K^{(j)})\quad\mbox{\rm in }{\cal
S}\times{\cal P}\times{\cal S},
\end{equation}
where $(\tilde Y^{(j)},\tilde Z^{(j)},\tilde K^{(j)})$  is a
solution of RBSDE in $\DD^j$ with terminal value $\xi^j$.
Moreover, by Proposition \ref{prop4.1} there exists $C>0$ such
that for $j,n\in\N$,
\begin{align*}
&E\Big(\sup_{t\leq T}|Y^{j,n}_t|^2
+\int_0^T\|Z^{j,n}_s\|^2\,ds+\sum_{s\leq T}|\Delta K^{j,n}_s|^2
+\int_0^T\dist(A^j_{s-},\partial D^j_{s-})\,d|K^{j,n}|_s\Big)\\
&\qquad\le C \Big(E\big(|\xi|^2
+\int_0^T|f(s,0,0)|^2\,ds\big)+\|A\|^2_{{\cal H}^2}\Big).
\end{align*}
Set  $\tau_{j,n}=\inf\{t;\sup_{s\leq t}|Y^{n}_s|>2N_j\}\wedge T$
and observe that by Tschebyshev's inequality and Proposition
\ref{prop4.1},
\[
P(\tau_{j,n}<T)\leq P(\sup_{t\leq T}|Y^{n}_t|>2N_j) \leq
(2N_j)^{-2}C \Big(E\big(\xi^2+\int_0^T|f(s,0,0)|^2\,ds\big)
+\|A\|_{{\cal H}^2}\Big),
\]
which implies that $\lim_{j\to\infty}\sup_nP(\tau_{j,n}<T)=0$. Let
$\sigma=\tau_{j,k}\wedge\gamma_j$.  Since $Y^n_t\in B(0,2N_j)$ for
$t<\sigma$, we may and will assume that
$\DD^{j,\sigma-}=\DD^{\sigma-}$. By Corollary \ref{cor4.3},
\begin{align*}
E\big(( \sup_{ t<\sigma}|Y^n_t-Y^{j,n}_t|^p\big)&
\leq CE(| Y^n_{\sigma-}-Y^{j,n}_{\sigma-}|^p)\\
&\leq E|\xi-\xi^{j}|^p
+E| Y^n_{\sigma-}-Y^{j,n}_{\sigma-}|^p{\bf 1}_{\{\sigma<T\}}\\
&\leq E|\xi-\xi^{j}|^p+(E|
Y^n_{\sigma-}-Y^{j,n}_{\sigma-}|^2)^{p/2}(P(\sigma<T))^{(2-p)/2}.
\end{align*}
Hence $\lim_{j\to\infty}\limsup_{n\to\infty}E(\sup_{t<\sigma}|
Y^n_{t}-Y^{j,n}_{t}|^p=0$. Consequently, for every
$\varepsilon>0$,
\begin{equation}
\label{eq4.17} \lim_{j\to\infty}\limsup_{n\to\infty}
P\big(\int_0^T\|Z^n_s-Z^{j,n}_s\|^2\,ds+\sup_{t\leq T}|
Y^n_{t}-Y^{j,n}_{t}|>\varepsilon\big)=0.
\end{equation}
The desired convergence follows from (\ref{eq4.19}),
(\ref{eq4.17}), because from Step 2 of the proof of Theorem
\ref{thm3.1} it follows that that $(\tilde Y^{(j)},\tilde
Z^{(j)},\tilde K^{(j)})\arrowp(Y,Z,K)$, where $(Y,Z,K)$ is a
unique solution of (\ref{eq1.1}).
\end{proof}

\begin{remark}
Using arguments from the proof of Step 2 of Theorem \ref{thm4.9}
one can show that in fact assumption (\ref{eq3.4}) in Proposition
\ref{prop4.6} (and consequently in the convergence statement in
Remark \ref{rem4.7}) is superfluous.
\end{remark}
\noindent{\bf Acknowledgements}\\
Research supported by Polish NCN grant no. 2012/07/B/ST1/03508.


\begin{thebibliography}{99}

\bibitem{BEO}
K. Bahlali, E. Essaky, Y. Ouknine, Reflected backward stochastic
differential equation with locally monotone coefficient, {\em
Stochastic Anal. Appl.} {\bf 22} (2004)  939--970.

\bibitem{CK}
J. Cvitanic, I. Karatzas, Backward stochastic differential
equations with reflection and Dynkin games, {\em  Ann. Probab.}
{\bf 24} (1996) 2024--2056.

\bibitem{GP}
A. Gegout--Petit, \'E. Pardoux, Equations diff\'erentielles
stochastiques r\'etrogrades r\'efl\'e\-chies dans un convexe, {\em
Stochastics Stochastics Rep.} {\bf 57} (1996) 111--128.

\bibitem{EKPPQ}
N. El Karoui, C. Kapoudjian, \'E. Pardoux, S.
Peng, M.C. Quenez, Reflected solutions of backward SDEs, and
related obstacle problems for PDE's, {\em Ann. Probab.} {\bf 25}
(1997) 702--737.

\bibitem{HHO}
S. Hamad\`ene, M. Hassani, Y. Ouknine, Backward SDEs with two rcll
reflecting barriers without Mokobodski's hypothesis, {\em  Bull.
Sci. Math.} {\bf 134} (2010) 874--899.

\bibitem{HO}
S. Hamad\`ene, Y. Ouknine, Reflected backward SDEs with general
jumps, {\em Theory Probab. Appl.} (to appear).

\bibitem{Kl:SPA}
T. Klimsiak, Reflected BSDEs and the obstacle problem for
semilinear PDEs in divergence form, {\em Stochastic Process.
Appl.} {\bf 122} (2012) 134--169.

\bibitem{Kl:EJP}
T. Klimsiak, Reflected BSDEs with monotone generator. {\em
Electron. J. Probab.} {\bf 17} (2012), no. 107, 1--25.

\bibitem{Kl:PA}
T. Klimsiak, Cauchy problem for semilinear parabolic equation with
time-dependent obstacles: a BSDEs approach, {\em Potential Anal.}
{\bf 39} (2013) 99--140.

\bibitem{Kl:BSM}
T. Klimsiak, BSDEs with monotone generator and two irregular
reflecting barriers, {\em Bull. Sci. Math.} {\bf 137} (2013)
268--321.

\bibitem{KR:JEE}
T. Klimsiak, A. Rozkosz, Obstacle problem for semilinear parabolic
equations with measure data (2013). arXiv:1301.5795

\bibitem{LX}
J.-P. Lepeltier, M. Xu, Penalization method for reflected backward
stochastic differential equations with one r.c.l.l. barrier, {\em
Statist. Probab. Lett.} {\bf 75} (2005) 58--66.

\bibitem{MPRZ}
L. Maticiuc, E. Pardoux, A.  R\v{a}\c{s}canu, A. Z\v{a}linescu,
Viscosity solutions for systems of parabolic variational
inequalities, {\em Bernoulli} {\bf 16} (2010)  258--273.

\bibitem{Me}
J.L. Menaldi,  Stochastic variational inequality for reflected
diffusion. {\em Indiana Univ. Math. J.} {\bf 32} (1983) 733--744.

\bibitem{O}
Y. Ouknine, Reflected backward stochastic differential equations
with jumps. {\em Stochastics Stochastics Rep.} {\bf 65} (1998)
111--125.

\bibitem{PP}
\'E. Pardoux, S. Peng, Adapted Solution of a Backward Stochastic
Differential Equation, {\em Systems Control Lett.} {\bf 14} (1990)
55--61.

\bibitem{PR}
E. Pardoux, A. R\v{a}\c{s}canu, Backward stochastic differential
equations with subdifferential operator and related variational
inequalities, {\em Stochastic Process. Appl.} {\bf 76} (1998)
191--215.

\bibitem{PX}
S. Peng, M. Xu,  The Smallest g-Supermartingale and Reflected BSDE
with Single and Double L2 obstacles, {\em  Ann. Inst. H.
Poincar\'e Probab. Statist.} {\bf 41} (2005) 605--630.

\bibitem{Pr}
Ph. Protter, Stochastic Integration and Differential Equations,
Springer-Verlag. Berlin Heidelberg 2005.

\bibitem{RS1}
A. Rozkosz, L. S\l omi\'nski, Stochastic representation of entropy
solutions of semilinear elliptic obstacle problems with measure
data, {\em Electron. J. Probab.} {\bf 17} (2012), no. 40, 1--27.

\bibitem{RS2}
A. Rozkosz, L. S\l omi\'nski, $L^p$ solutions of reflected BSDEs
under monotonicity condition, {\em Stochastic Process. Appl.} {\bf
122} (2012) 3875--3900.

\end{thebibliography}
\end{document}